\documentclass[12pt,twoside]{amsart}

\usepackage{amsfonts,amsmath,amsthm,amssymb,latexsym,amsthm,newlfont,enumerate,color,amscd}

\newif\ifslide

\theoremstyle{plain}

\ifdefined\spacer
\newtheorem{theorem}{Theorem}
 \else
\newtheorem{theorem}{Theorem}[section]
\fi

\newtheorem{lemma}[theorem]{Lemma}

\newtheorem*{theorem*}{Theorem}

\newtheorem{proposition}[theorem]{Proposition}
\newtheorem{definition-lemma}[theorem]{Definition-Lemma}

\newtheorem{red-question}[theorem]{\textcolor{red}{Question}}

\theoremstyle{definition}
\newtheorem{definition}[theorem]{Definition}
\newtheorem{remark}[theorem]{Remark}

\def\ideal#1.{I_{#1}}
\def\ring#1.{\mathcal {O}_{#1}}
\def\Spec{\operatorname {Spec}}

\def\fring#1.{\hat{\mathcal {O}}_{#1}}
\def\proj#1.{\mathbb {P}(#1)}
\def\pr #1.{\mathbb {P}^{#1}}
\def\dpr #1.{\hat{\mathbb {P}}^{#1}}
\def\af #1.{\mathbb A^{#1}}
\def\Hz #1.{\mathbb F_{#1}}
\def\Hbz #1.{\overline{\mathbb F}_{#1}}
\def\fb#1.{\underset #1 {\times}}
\def\rest#1.{\underset {\ \ring #1.} \to \otimes}
\def\au#1.{\operatorname {Aut}\,(#1)}
\def\deg#1.{\operatorname {deg } (#1)}

\def\pic#1.{\operatorname {Pic}\,(#1)}
\def\pico#1.{\operatorname{Pic}^0(#1)}
\def\picg#1.{\operatorname {Pic}^G(#1)}
\def\ner#1.{NS (#1)}
\def\rdown#1.{\llcorner#1\lrcorner}
\def\rfdown#1.{\lfloor{#1}\rfloor}
\def\rup#1.{\ulcorner{#1}\urcorner}
\def\rcup#1.{\lceil{#1}\rceil}

\def\n1#1.{\operatorname {N_1}(#1)}  
\def\cn1#1.{\overline{\operatorname {N^1}(#1)}} 
\def\cone#1.{\operatorname {NE}(#1)}     
\def\ccone#1.{\overline{\operatorname {NE}}(#1)}
\def\none#1.{\operatorname {NF}(#1)}
\def\cnone#1.{\overline{\operatorname {NF}}(#1)}
\def\mone#1.{\operatorname {NM}(#1)} 
\def\cmone#1.{\overline{\operatorname {NM}}(#1)}

\def\coef#1.{\frac{(#1-1)}{#1}}
\def\vit#1.{D_{\langle #1 \rangle}}
\def\mm#1.{\overline {M}_{0,#1}}
\def\H1#1.{H^1(#1,{\ring #1.})}
\def\ac#1.{\overline {\mathbb F}_{#1}}

\def\adj#1.{\frac {#1-1}{#1}}
\def\spn#1.{\overline{#1}}
\def\pek#1.#2.{\Cal P^{#1}(#2)}
\def\plk#1.#2.{\Cal P^{\leq #1}(#2)}
\def\ev#1.{\operatorname{ev_{#1}}}
\def\ilist#1.{{#1}_1,{#1}_2,\dots}
\def\bminv#1.{(\nu_1,s_1;\nu_2,s_2;\dots ;\nu_{#1},s_{#1};\nu_{r+1})}
\def\zinv#1.{(\nu_1,s_1;\nu_2,s_2;\dots ;\nu_{#1},s_{#1};0)}
\def\iinv#1.{(\nu_1,s_1;\nu_2,s_2;\dots ;\nu_{#1},s_{#1};\infty)}

\def\scr #1.{\mathcal #1}

\def\llist#1.#2.{{#1}_1,{#1}_2,\dots,{#1}_{#2}}
\def\ulist#1.#2.{{#1}^1,{#1}^2,\dots,{#1}^{#2}}
\def\lomitlist#1.#2.{{#1}_1,{#1}_2,\dots,\hat {{#1}_i}, \dots, {#1}_{#2}}
\def\lomitlistz#1.#2.{{#1}_0,{#1}_1,\dots,\hat {{#1}_i}, \dots, {#1}_{#2}}
\def\loc#1.#2.{\Cal O_{#1,#2}}
\def\fderiv#1.#2.{\frac {\partial #1}{\partial #2}}
\def\deriv#1.#2.{\frac {d #1}{d #2}}

\def\map#1.#2.{#1 \longrightarrow #2}
\def\rmap#1.#2.{#1 \dasharrow #2}
\def\emb#1.#2.{#1 \hookrightarrow #2}
\def\non#1.#2.{\text {Spec }#1[\epsilon]/(\epsilon)^{#2}}
\def\Hi#1.#2.{\text {Hilb}^{#1}(#2)}
\def\sym#1.#2.{\operatorname {Sym}^{#1}(#2)}
\def\Hb#1.#2.{\text {Hilb}_{#1}(#2)}
\def\Hm#1.#2.{\Hom_{#1}(#2)}
\def\prd#1.#2.{{#1}_1\cdot {#1}_2\cdots {#1}_{#2}}
\def\Bl #1.#2.{\operatorname {Bl}_{#1}#2}
\def\pl #1.#2.{#1^{\otimes #2}}
\def\mgn#1.#2.{\overline {M}_{#1,#2}}
\def\ialist#1.#2.{{#1}_1 #2 {#1}_2, #2\dots}
\def\pair#1.#2.{\langle #1, #2\rangle}
\def\vandermonde#1.#2.{\left|
\begin{matrix}
1 & 1 & 1 & \dots & 1\\
{#1}_1 & {#1}_2 & {#1}_3 & \dots & {#1}_{#2}\\
{#1}_1^2 & {#1}_2^2 & {#1}_3^2 & \dots & {#1}_{#2}^2\\
\vdots & \vdots & \vdots & \ddots & \vdots\\
{#1}_1^{#2-1} & {#1}_2^{#2-1} & {#1}_2^{#2-1} & \dots & {#1}_{#2}^{#2-1}\\
\end{matrix}
\right|
}
\def\vandermondet#1.#2.{\left|
\begin{matrix}
1 & {#1}_1   & {#1}_1^2 & \dots & {#1}_1^{#2-1}\\
1 & {#1}_2   & {#1}_2^2 & \dots & {#1}_2^{#2-1}\\
1 & {#1}_3   & {#1}_3^2 & \dots & {#1}_3^{#2-1}\\
\vdots & \vdots & \vdots & \ddots & \vdots\\
1 & {#1}_{#2}& {#1}_{#2}^2 & \dots & {#1}_{#2}^{#2-1}\\
\end{matrix}
\right|
}
\def\gr#1.#2.{\mathbb{G}(#1,#2)}

\def\alist#1.#2.#3.{{#1}_1 #2 {#1}_2 #2\dots #2 {#1}_{#3}}
\def\zlist#1.#2.#3.{#1_0 #2 #1_1 #2\dots #2 #1_{#3}}
\def\lomitlist30#1.#2.#3.{{#1}_0,{#1}_1 #2 \dots #2\hat {{#1}_i} #2\dots #2 {#1}_{#3}}
\def\lmap#1.#2.#3.{#1 \overset{#2}{\longrightarrow} #3}
\def\mes#1.#2.#3.{#1 \longrightarrow #2 \longrightarrow #3}
\def\ses#1.#2.#3.{0\longrightarrow #1 \longrightarrow #2 \longrightarrow #3 \longrightarrow 0}
\def\les#1.#2.#3.{0\longrightarrow #1 \longrightarrow #2 \longrightarrow #3}
\def\res#1.#2.#3.{#1 \longrightarrow #2 \longrightarrow #3\longrightarrow 0}
\def\Hi#1.#2.#3.{\text {Hilb}^{#1}_{#2}(#3)}
\def\ten#1.#2.#3.{#1\underset {#2}{\otimes} #3}
\def\lomitlist30#1.#2.#3.{{#1}_0 #2 {#1}_1 #2 \dots #2 \hat {{#1}_i} #2 \dots #2 {#1}_{#3}}
\def\mderiv#1.#2.#3.{\frac {d^{#3} #1}{d #2^{#3}}}

\def\Diff{\operatorname{Diff}}

\def\Hom{\operatorname{Hom}}

\def\Supp{\operatorname{Supp}}

\def\Exc{\operatorname{Exc}}

\def\dim{\operatorname{dim}}

\def\deg{\operatorname{deg}}

\def\det{\operatorname{det}}

\def\Sing{\operatorname{Sing}}

\def\Hilb{\operatorname{Hilb}}
\def\Div{\operatorname{Div}}

\def\rest{\operatorname{res}}

\def\e{\Cal E}

\def\e1{E_1}
\def\e2{E_2}

\def\Q{\mathbb Q}
\def\Z{\mathbb Z}

\def\mapdown#1{\big\downarrow\rlap{$\vcenter{\hbox{$\scriptstyle#1$}}$}}

\def\mapse#1{
{\vcenter{\hbox{$\mathop{\smash{\raise1pt\hbox{$\diagdown$}\!\lower7pt
\hbox{$\searrow$}}\vphantom{p}}\limits_{#1}\vphantom{\mapdown{}}$}}}}

\def\VR#1.{height#1pt&\omit&&\omit&&\omit&&\omit&&\omit&\cr}

\def\VRT#1.{height#1pt&\omit&&\omit&\cr}

\usepackage{hyperref}

\usepackage{mathtools}
\usepackage{tikz}
\usepackage{tikz-cd}
\usetikzlibrary{matrix,arrows,decorations.pathmorphing, shapes, calc}

\usepackage{comment}
\usepackage[normalem]{ulem} 
\usepackage{ marvosym }

\title[On log del Pezzo surfaces]
{On  log del Pezzo surfaces in large characteristic} 
\author{Paolo Cascini, Hiromu Tanaka, Jakub Witaszek} 
\thanks{All three authors were funded by EPSRC}
\address{Department of Mathematics, Imperial College, London, 180 Queen's Gate, 
London SW7 2AZ, UK} 
\email{p.cascini@imperial.ac.uk}
\email{h.tanaka@imperial.ac.uk}
\email{j.witaszek14@imperial.ac.uk}

\newtheorem{step}{Step}

\usepackage{mathtools}

\begin{document}

\maketitle

\begin{abstract}
We show that any Kawamata log terminal del Pezzo surface  over an algebraically closed field of large characteristic is 
globally $F$-regular or it admits a log resolution which lifts to characteristic zero. 
As a consequence, we prove the Kawamata--Viehweg vanishing theorem for klt del Pezzo surfaces of large characteristic.
\end{abstract}

\tableofcontents

\setcounter{section}{0}

\section{Introduction}

Recently, starting from the work of Hacon and Xu \cite{hx13}, many of the classical results of the
 minimal model programme in characteristic zero have been extended to  three dimensional varieties over an algebraically closed field $k$ of  characteristic $p>5$  \cite{birkar13,ctx13,bw14}.
One of the main tools used in \cite{hx13} is the theory of $F$-singularities, replacing the use of the classical vanishing theorems, which hold only in characteristic zero (e.g.\ see \cite{schwede09}). 
Their proof of the existence of flips relies on 
the fact that if $\mathrm{char}\,k>5$ and $(X, \Delta)$ is a one-dimensional projective klt pair such that 
$-(K_X+\Delta)$ is ample and the coefficients of $\Delta$ are contained in the standard set $\{1-\frac{1}{n}\,|\, n \in \mathbb Z_{>0}\}$, 
then $(X, \Delta)$ is globally $F$-regular \cite[Theorem 4.2]{watanabe91}.
Thus, it is natural to ask whether this result can be generalised to higher dimensional varieties.
 Unfortunately, in \cite{CTW15a} we give a negative answer to this question. 
Indeed, we show that  
over an arbitrary algebraically closed field $k$ of characteristic $p>0$, 
there exists a projective klt surface $X$ over $k$ such that $-K_X$ is ample, but $X$ is not globally $F$-regular. 
Thus, even in large characteristic, it is not known a priori whether klt del Pezzo surfaces admit
 desirable properties, 
such as  the Kawamata--Viehweg vanishing theorem and the Bogomolov bound, which in particular gives a bound on the number of singular points on a klt del Pezzo surface of Picard number one (\cite{langer15} and \cite[Section 9]{km99}). 

\medskip

The goal of this paper is to show the following:

\begin{theorem}\label{thm-main}
Let $I\subseteq (0,1)\cap \mathbb Q$ be a finite set. 

Then there exists a positive integer $p(I)$ which satisfies the following property: 

Let $(X, B)$ be a two dimensional projective klt pair over an algebraically closed field of characteristic $p>p(I)$ 
such that $-(K_X+B)$ is ample and the coefficients of $B$ are contained in $I$. 

Then one of the following properties holds:
\begin{enumerate}
\item{$(X, B)$ is globally $F$-regular, or}
\item{there exists a log resolution $\mu \colon V \to X$ of $(X, B)$ 
such that $(V, \Exc(\mu) \cup \mu^{-1}_*(\Supp B))$ lifts to characteristic zero over a smooth base 
(cf.\ Definition~\ref{d-liftable}). }
\end{enumerate}
\end{theorem}
Note that we do not know whether there exists a  klt  del Pezzo surface in large characteristic which violates (2) of Theorem \ref{thm-main}.

\medskip

Using the above result, the minimal model programme 
and the logarithmic version of a result of Deligne--Illusie \cite[Corollary 3.8]{hara98a}, 
we prove the Kawamata--Viehweg vanishing theorem for klt del Pezzo surfaces in large characteristic:

\begin{theorem}\label{t_kv}
There exists a positive integer $p_0$ which satisfies the following property: 

Let $(X, \Delta)$ be a two dimensional projective klt pair over an algebraically closed field of characteristic $p>p_0$. 
Suppose that there exists an effective $\Q$-divisor $B$ such that $(X, B)$ is klt and $-(K_X+B)$ is nef and big. 
If $D$ is a 
$\Z$-divisor  on $X$ such that $D-(K_X+\Delta)$ is nef and big, then 
\[
H^i(X, \mathcal O_X(D))=0
\qquad\text{for any $i>0$}. 
\]
\end{theorem}

Note that Theorem \ref{thm-main} imposes a condition on the coefficients of the boundary divisor $B$, 
but Theorem \ref{t_kv} does not need such an assumption. 

Finally, we show that Theorem~\ref{thm-main} does not hold in characteristic two. 

\begin{theorem}\label{t_char2_example}
There exists a projective klt surface $X$ over $\overline{\mathbb F}_2$ 
which satisfies the following properties:
\begin{enumerate}
\item{$-K_X$ is ample,}
\item{$X$ is not globally $F$-split, and }
\item{for any log resolution $h\colon W \to X$ of $X$, the pair
$(W, \Exc(h))$ does not lift to characteristic zero over a smooth base  (cf.\ Definition~\ref{d-liftable}). }
\end{enumerate}
\end{theorem}


%
%

\medskip

We prove Theorem \ref{thm-main} in Section \ref{s_proof-main-thm}, Theorem \ref{t_kv} in Section \ref{s_kv} and Theorem 
\ref{t_char2_example} in Section \ref{s_char2_example}.

\subsection{Sketch of the proof } 
We now give an overview of some of the steps of the  proof of Theorem~\ref{thm-main}.

For simplicity, we assume that $I=\emptyset$. 
Let $X$ be a projective klt surface in large characteristic and such that $-K_X$ is ample. 
We want to show that at least one of the statements (1) and (2) in Theorem~\ref{thm-main} holds true. 
The idea  is that, choosing a suitable $\epsilon \in \Q_{>0}$, we consider two different cases:  
\begin{enumerate}
\item[(a)]{$X$ is $\epsilon$-klt.}
\item[(b)]{$X$ is not $\epsilon$-klt.}
\end{enumerate}

\medskip

(a) 
It is known that, for a fixed algebraically closed field $k$,  
the set of all the $\epsilon$-klt del Pezzo surfaces over $k$ forms a bounded family \cite{alexeev94}. 
We  generalise this boundedness result to mixed characteristic, i.e.\ 
there exists a projective morphism $\mathcal X \to S$ of schemes of finite type over $\Spec\,\Z$, 
depending only on $\epsilon$,  
such that an arbitrary $\epsilon$-klt del Pezzo surface $X$ over an arbitrary algebraically closed field $k$ of characteristic $p>5$ 
is isomorphic to some fibre $\mathcal X_s$ up to a base change of the base field. 
By  noetherian induction, 
we can show that for $p \gg 0$, 
any $\epsilon$-klt del Pezzo surface $X$ over an algebraic closed field of characteristic $p$ 
admits a log resolution $\mu\colon V \to X$ such that 
$(V, \Exc(\mu))$ lifts to characteristic zero over a smooth base. 
For more details, see Section~\ref{section-bounded-case}. 

\medskip

(b) 
Assume that $X$ is not $\epsilon$-klt. 
In this case, the proof consists of the following three steps:

\begin{enumerate}
\item[(I)]{Since $X$ is not $\epsilon$-klt, 
we can extract a prime divisor $C$ with  log discrepancy $a:=a(C, X, 0)\leq \epsilon$ by: 
\[
f\colon Y \to X,\quad\text{so that} \quad K_Y+(1-a)C=f^*K_X.
\]
}
\item[(II)]{We run a $-(K_Y+C)$-MMP with scaling of $C$: 
\[
(Y, C)=:(Y_0, C_0) \to (Y_1, C_1) \to \cdots \to (Y_n, C_n)=:(Z, C_Z).
\]
Note that since  $(Y, \Delta)$ is log del Pezzo for some effective $\mathbb Q$-divisor $\Delta$, we may run a $D$-MMP for any $\mathbb Q$-Cartier divisor $D$. 
Indeed, let $s>0$ be a sufficiently small rational number such that $A:=sD-(K_Y+\Delta)$ is ample. Then,
\[
sD=K_Y+\Delta +A
\] 
and in particular, a $D$-MMP coincides with a $(K_Y+\Delta+A)$-MMP. 
Furthermore, if
 $\epsilon$ is sufficiently small, then since $1-a\in [1-\epsilon,1)$, it follows that 
each pair $(Y_i, C_i)$ is log canonical by ACC for the log canonical threshold \cite[Theorem 1.1]{alexeev93}.}
\item[(III)]{We now consider five different cases as follows: 
\begin{enumerate}
\item[(i)]{$\kappa(Z, -(K_Z+C_Z))=2$.}
\item[(ii)]{$\kappa(Z, -(K_Z+C_Z))=1$.}
\item[(iii)]{$\kappa(Z, -(K_Z+C_Z))=0$ and $(Z, C_Z)$ is not plt.}
\item[(iv)]{$\kappa(Z, -(K_Z+C_Z))=0$ and $(Z, C_Z)$ is plt.}
\item[(v)]{$\kappa(Z, -(K_Z+C_Z))=-\infty$.}
\end{enumerate}
If one of the cases (i), (ii), or (iii) holds, then we can show that $X$ is globally $F$-regular, 
i.e.\ (1) in Theorem~\ref{thm-main} holds. 
If (iv) holds, then (2) in Theorem~\ref{thm-main} holds. 
Finally, we  show that (v) does not occur if $\epsilon$ is sufficiently small. }
\end{enumerate}

\medskip

We now give some details of the methods we use in the  steps above. After Step (I), 
there exists a rational number $b\in (1-a,1)$, such that  $(Y, bC)$ is klt and $-(K_Y+bC)$ is ample. 
The easiest case  is if $-(K_Y+C)$ is ample and $(p^e-1)(K_Y+C)$ is Cartier for some positive integer $e$. 
In this case, Lemma \ref{l_gladj} implies immediately that  $(Z, C_Z)$ is globally $F$-split, 
which in turn implies that $X$ is globally $F$-regular. 
However, $-(K_Y+C)$ is not ample in general, but by running an MMP in Step (II), 
we can get closer to this situation. 
Indeed, in case (i) of Step (III), we have that $-(K_Z+C_Z)$ is nef and big, and after contracting the curves $\Gamma$ with $(K_Z+C_Z)\cdot \Gamma=0$, 
we may assume that $-(K_Z+C_Z)$ is ample. 
By using Lemma \ref{l_gladj}, we can show that $X$ is globally $F$-regular. 
We can apply a similar argument to show that $X$ is globally $F$-regular in cases (ii) and (iii) of Step (III). 

Let us assume now that (iv) of Step (III) holds. 
In this case, we first show that 
the set of pairs $(Z, C_Z)$ as in (iv) forms a bounded family. 
By  noetherian induction, 
we  show that for $p \gg 0$, 
any $(Z, C_Z)$ has a log resolution $\pi \colon W \to Z$ such that 
$(W, \pi_*^{-1}C_Z \cup \Exc(\pi))$ lifts to characteristic zero over a smooth base. 
If $W$ already dominates $Y$, then the induced birational morphism $W \to X$ 
is a log resolution as in (2) of Theorem~\ref{thm-main}. 
In general, we cannot hope that $W$ dominates $Y$, however we can show that (2) of Theorem~\ref{thm-main} holds after taking some blow-ups of $W$. For more details, we refer to  Section 4, especially Lemma~\ref{l_blow_up_lift}.

Finally, let us assume that (v) of Step  (III) holds. 
In particular,  $Z$ admits a $-(K_Z+C_Z)$-negative Mori fibre space. Let us assume for simplicity that $\rho(Z)=1$ and $K_Z+C_Z$ is ample. 
Since $-(K_Y+bC)$ is ample, so is its push-forward $-(K_Z+bC_Z)$. 
Thus, we may find $b'\in (b,1)$ such that 
$$K_Z+b'C_Z \equiv 0.$$
In particular, $b'\in (1-\epsilon, 1)$, and after possibly replacing $\epsilon$ by a smaller number, we derive a contradiction thanks to  ACC for the log Calabi--Yau pairs \cite[Theorem 5.3]{alexeev93} (cf.\ Lemma \ref{l_acc.num}).

\medskip

\textbf{Acknowledgement: } 
We would like to thank P.\ Achinger, Y.\ Gongyo, Y.\ Kawamata, A.\ Langer, C.\ Liedtke, Z.\ Patakfalvi, A.\ Sannai, K.\ Schwede and S.\ Takagi for many useful discussions and comments. 
We would like to thank the referee for
carefully reading our manuscript and for suggesting several improvements.

\section{Preliminaries}

If not stated otherwise, we work over an algebraically closed field $k$ of characteristic $p>0$. 
We say that $X$ is a {\em variety} over an algebraically closed field $k$, if $X$ is an integral  scheme which is separated and of finite type over $k$. 
A {\em curve} is a variety of dimension one, and a {\em surface} is a variety of dimension two. 
We say that a scheme $X$ is {\em normal} if the local ring $\mathcal O_{X, x}$ is an integrally closed integral domain for any point $x \in X$. 
In particular, if $X$ is a noetherian normal scheme, then $\Gamma(U, \mathcal O_X)$ is an integrally closed integral domain 
for any irreducible affine open subset $U$ of $X$. Given a proper morphism $f\colon X\to Y$ between normal varieties, we say that  two 
$\mathbb Q$-Cartier $\mathbb Q$-divisors 
$D_1,D_2$ on $X$ are {\em numerically equivalent over} $Y$, denoted $D_1\equiv_f D_2$, if their difference is numerically trivial on any fibre of $f$.  

We refer to \cite{KM98} for the classical definitions of singularities (e.g., {\em klt, plt, log canonical}) appearing in the minimal model programme. 
A {\em log pair} $(X,\Delta)$ consists of a normal variety $X$ and a $\mathbb Q$-divisor $\Delta\ge 0$ such that $K_X+\Delta$ is $\mathbb Q$-Cartier. 
Note  that we always assume that a klt (resp.\ plt, log canonical) pair $(X, \Delta)$ 
 is  a log pair, and in particular $\Delta$ is an effective $\mathbb Q$-divisor. 
Given a log pair $(X,B)$ and a divisorial valuation $E$ over $X$, we denote by $a(E,X,B)$ the {\em log discrepancy} of $(X,B)$ with respect to $E$.
For $\varepsilon>0$, we say that $(X,B)$ is $\varepsilon$-{\em klt} if $a(E,X,B)>\varepsilon$ for any divisorial valuation $E$ over $X$. 
A two dimensional projective log pair $(X, \Delta)$ is {\em log del Pezzo} if 
$(X, \Delta)$ is klt and $-(K_X+\Delta)$ is ample. 

Given a subset $I\subseteq [0,1]$, we say that $I$ is an {\em ACC} (resp.\ {a \em DCC}) set if
it satisfies the ascending chain condition (resp.\ the descending chain
condition). Given a subset  $I\subseteq [0,1]$, we define:
\[
I_+:=\Big\{ { \sum_{j=1}^m  i_j} \; \big| \; i_j\in I \text{ for } j=1,\dots,m \Big\}\cap [0,1],
\]
and
\[
D(I):=\Big\{ \frac{m-1+f}{m} \; \mid \; m\in \mathbb Z_{>0}, f\in I_+\Big\}\cap [0,1].
\]

For the definitions of $F$-singularities, 
we refer to \cite[Definition 3.1]{schwedesmith10} and \cite[Definition 1.6]{CTW15a}.

\subsection{Log canonical surface singularities}\label{s_diff}

Let $X$ be a normal variety over an algebraically closed field $k$. 
Let $C$ be a prime divisor on $X$ such that $K_X+C$ is $\mathbb Q$-Cartier and let 
 $C' \to C$ be its normalisation. 
Then there exists an effective $\mathbb Q$-Cartier $\mathbb Q$-divisor $\Diff_{C'}$ on $C'$ such that 
\[
(K_X+C)|_{C'}=K_{C'}+\Diff_{C'}.
\]
Moreover, if $(X, C)$ is log canonical, then 
the coefficients of $\Diff_{C'}$ are standard coefficients, 
i.e.\ they are contained in the set $\{1\} \cup \{1-\frac{1}m\mid m\in \mathbb Z_{>0}\}$ 
(see \cite[Proposition-Definition 16.5]{Kollaretal} and \cite[Definition 4.2 and Proposition 4.5]{kollar13}).

Let $X$ be a normal surface. A singular point $q\in X$  is said to be  {\em cyclic} if the exceptional locus of the minimal resolution of $X$ at $q$ is a chain of rational curves $C_1,\dots,C_m$. 
We denote by $\Gamma_q$ the intersection matrix $(-C_i \cdot C_j)$ associated to $C_1,\dots,C_m$. In particular, if $d$ is a positive integer, then  there are only finitely many possibilities for $\Gamma_q$ so that $\det \Gamma_q=d$ \cite[pag.\ 116]{kollar13}.

A singular point $q\in X$ is said to be {\em dihedral} if the exceptional locus of the minimal resolution of $X$ at $q$ is the union of rational curves $C_1,\dots,C_m$, with $m\ge 4$,  where $C_1,\dots,C_{m-2}$ is a chain of curves and $C_{m-1}, C_m$ are $(-2)$-curves which meet  $C_{m-2}$ transversally in two distinct points away from $C_{m-3}\cap C_{m-2}$ \cite[3.35(3)]{kollar13}.

\begin{lemma}\label{l_klt_Cartier}
Let $f \colon X \to Y$ be a proper birational morphism between klt surfaces. 
If $L$ is a Cartier divisor on $X$ such that $L \equiv_f 0$, then 
there exists a Cartier divisor $L_Y$ on $Y$ such that $\mathcal O_X(L) \simeq \mathcal O_X(f^*L_Y)$. 
In particular, if $D$ is a Weil divisor on $Y$ such that  $f^*D$ is Cartier, then  $D$ is  Cartier. 
\end{lemma}

\begin{proof}
Since a klt surface is automatically quasi-projective (e.g.\ see \cite[Lemma 2.2]{fujino12}),
by taking suitable compactifications, we may assume that $X$ and $Y$ are projective.
Let  $h\colon \tilde{X}\to X $ be  the minimal resolution and let $g\colon \tilde X\to Y$ be the induced morphism. It is enough to show that there exists a Cartier divisor $L_Y$ on $Y$ such that  $\mathcal O_{\tilde X}(h^*L)\simeq 
\mathcal O_{\tilde X}(g^*L_Y)$. Let $E$ be the sum of the exceptional prime divisors of $g$. 
By running a $(K_{\tilde X}+E)$-MMP over $Y$, we may assume that $f$ is a birational morphism between projective klt surfaces such that $\rho(X/Y)=1$.

By using the base point free theorem over surfaces \cite[Theorem 3.2 and Corollary 3.6]{tanaka12a}, 
we may apply the same argument as \cite[Theorem 3.7(4)]{KM98} to conclude. 
\end{proof}

\begin{lemma}\label{l_diff}
Let $(X, C)$ be a two dimensional plt pair, where $C$ is a prime divisor. 
If $q \in C$ is a singular point of $X$, 
then $q$ is cyclic and  the coefficient of  $\Diff_C$ at $q$ is given by 
$1-\frac1{m_q}$, where $m_q:=\det(\Gamma_q).$ 
Further, for each Weil divisor $D$ on $X$, we have that $m_qD$ is Cartier around $q$. 
\end{lemma}

\begin{proof} The first part of the lemma follows from \cite[Theorem 3.36]{kollar13}. We are left to show that if $D$ is a Weil divisor on $X$, then $m_qD$ is Cartier at $q$. Let $f\colon Y\to X$ be the minimal resolution of $X$ at $q$. Then the exceptional divisor is a chain of rational curves $C_1,\dots,C_r$. We may write 
\[
f^*D=D_Y+\sum_{i=1}^r b_i C_i,
\]
 where $D_Y$ is the strict transform of $D$ in $Y$ and  $b_1,\dots,b_r\in \mathbb Q$.
Thus, for any $j=1,\dots,r$, we have 
\[
\sum_{i=1}^r b_i C_i\cdot C_j=-D_Y\cdot C_j.
\] 
Since $D_Y\cdot C_j\in \mathbb Z$, it follows that $m_q b_i\in \mathbb Z$ for each $i=1,\dots,r$. Since $X$ is klt at $q$, Lemma \ref{l_klt_Cartier} implies the claim.  
\end{proof}

%

We   need the following version of ACC for the nef threshold in dimension one and two:

\begin{lemma}\label{l_acc.num}
Let $I\subseteq (0,1]\cap \mathbb Q$ be a DCC set. Then there exists a finite subset $J\subseteq I$
 which satisfies the following property:

Let $(X,B)$ be a projective log canonical pair and let  $\pi \colon X \to T$ be a projective morphism onto a normal variety $T$ such that $\pi_*\mathcal O_X=\mathcal O_T$ and
\begin{itemize}
\item $\dim X\le 2$,
\item $\dim T<\dim X$, 
\item the coefficients of $B$ belong to $I$, and 
\item $K_X+B\equiv_\pi 0$.
\end{itemize}
Then the coefficient of any $\pi$-horizontal component of $B$ belongs to $J$.  

\end{lemma}
\begin{proof}
See \cite[Theorem 5.3]{alexeev93} and \cite[Proposition 11.7]{birkar13}.
\end{proof}

To show Proposition~\ref{p_bounded-index}, we need the following:

\begin{lemma}\label{l_klt_index}
Let $(X, \Delta)$ be a two dimensional klt pair such that $X$ is smooth away from a closed point $x\in X$. Let $f\colon Y \to X$ be the minimal resolution of $X$, with  exceptional divisors $E_1,\dots,E_n$ and let $a_i=a(E_i,X,\Delta)$. 
Assume that $n\ge 8$ and let $\sigma$ be a permutation of $\{1,\dots,n\}$ such that  
\[
a_{\sigma(1)}\le a_{\sigma(2)}\le\dots\le a_{\sigma(n)}.
\]
Let $\ell$ be a positive integer such that $\ell\Delta$ is a $\mathbb Z$-divisor and 
\[
\ell a_{\sigma(i)}\in \mathbb Z\qquad\text{for }i=1,\dots,8.
\]

Then $2\ell(K_X+\Delta)$ is Cartier. 
\end{lemma}

\begin{proof}
We have
\[
K_Y+f_*^{-1}\Delta+\sum_{i=1}^n (1-a_i)E_i=f^*(K_X+\Delta).
\]
Since $X$ is klt and $\ell f_*^{-1}\Delta$ is integral, by Lemma \ref{l_klt_Cartier}, it is enough to show that $2\ell a_j \in \mathbb Z$ for any $j=1,\dots, n$. 

We may assume that $x \in X$ is a cyclic or a dihedral singularity, as otherwise, by the classification of two dimensional klt singularities, we have that $n\le 8$ (e.g.\ see \cite[Theorem 4.16]{KM98}), and the claim follows immediately.

We first assume that 
$x\in X$ is a cyclic singularity and $E_1,\dots, E_n$ is a chain of rational curves. 
For each $i \in \{2,\dots,n-1\}$, by taking the intersection with  $E_i$ on both sides of the equality above, 
we obtain 
\[
(-E_i^2)a_i+f_*^{-1}\Delta \cdot E_i=a_{i-1}+a_{i+1}.
\]
In particular, the numbers $a_1,\dots,a_n$ satisfy the following convexity inequality: 
\[
a_i \leq \frac{a_{i-1}+a_{i+1}}{-E_i^2} \leq \frac{a_{i-1}+a_{i+1}}{2}.
\]
Thus, after possibly replacing $\sigma(1)$ by another index $j\in \{1,\dots,n\}$ such that $a_{\sigma(1)}=a_{j}$, 
we may assume that $\sigma(2)$ is equal to $\sigma(1)-1$ or $\sigma(1)+1$, 
say $\sigma(1)+1$. 
In particular, 
$\ell a_{\sigma(1)}, \ell a_{\sigma(1)+1} \in \mathbb Z$.

Thus, for each $i=2,\dots,n-1$, we have
\[
(-E_i^2)\ell a_i+(\ell f_*^{-1}\Delta) \cdot E_i=\ell a_{i-1}+\ell a_{i+1}, 
\]
and it follows inductively that $\ell a_j\in \mathbb Z$ for any $j=1,\dots,n$. Thus,  the claim follows. 
\medskip

Let us assume now that $x\in X$ is a dihedral singularity. 
Let $E_{n-1}$ and $E_n$ be two tails of self-intersection $(-2)$ 
and let $E_1,E_2,\dots,E_{n-2}$ be the remaining chain of rational curves \cite[3.35(3)]{kollar13}, 
so that  $E_{n-2}$ intersects $E_{n-1}$ and $E_n$.  
For $m \in \{n-1, n\}$, by taking the intersection with $E_m$ on both sides of the equality 
\[
K_Y+f_*^{-1}\Delta+\sum_{i=1}^n (1-a_i)E_i=f^*(K_X+\Delta),
\]
we obtain 
\[
2a_m+f_*^{-1}\Delta \cdot E_m=a_{n-2}+1.
\]

Thus, it is enough to show that $\ell a_i\in \mathbb Z$ for any $i=1,\dots,n-2$. 
By  assumption, it follows that 
\[
\ell \min_{1\leq i \leq n-2}a_i,\,\, \ell \min_{\substack{1\leq i \leq n-2\\ i \neq j_1}}a_i \in \mathbb Z
\]
for some $j_1\in \{1,\dots,n-2\}$ such that $a_{j_1}=\min_{1\leq i \leq n-2}a_i$. 
By applying the same argument as above to the chain 
$E_1,E_2,\dots,E_{n-2},$
we obtain that $\ell a_j \in \mathbb Z$ for every $j\in \{1,\dots,n-2\}$. 
Thus, the claim follows. 
\end{proof}

\begin{proposition}\label{p_bounded-index}
Let $I \subseteq (0, 1] \cap \mathbb Q$ be a DCC set. 
Then there exists a positive integer $m_0$ which satisfies the following property:

Let $(X, \Delta)$ be a two dimensional projective log canonical pair such that   
\begin{itemize}
\item{$X$ is klt, }
\item{the coefficients of $\Delta$ are contained in $I$, and }
\item{$K_X+\Delta \equiv 0$. }
\end{itemize}
Then $m_0(K_X+\Delta)$ is Cartier. 
\end{proposition}

\begin{proof}
By Lemma \ref{l_acc.num}, we may assume that $I$ is a finite set. Further, we may assume that $1 \in I$. Let $(X,\Delta)$ be a log pair as in the Proposition. 

First, we show that there exists  $m_1 \in \mathbb Z_{>0}$, depending only on $I$, such that 
$m_1(K_X+\Delta)$ is Cartier around any closed point $q\in X$ such that
$(X, \Delta)$ is not klt at $q$. 
By Lemma \ref{l_klt_Cartier}, after taking a dlt modification, we may assume that $(X, \Delta)$ is dlt at $q$. 
If  the support of $\rdown \Delta .$ is singular at $q$, 
then $X$ is smooth at $q$ \cite[Theorem~4.15(1)]{KM98}, and there is nothing to show. 
Thus, we may assume that $(X, \Delta)$ is plt around $q$. Let $C$ be the component of $\lfloor \Delta \rfloor$ containing $q$. Since $(X,\Delta)$ is plt around $q$, we know that $C$ is smooth, and so we may write  
\[
K_C+ \Delta_C=(K_X+\Delta)|_C\equiv 0, 
\]
where, by inversion of adjunction,  $\Delta_C$ is a $\mathbb Q$-divisor on $C$ such that $(C,\Delta_C)$ is klt. By \cite[Lemma 4.3]{mp04}, $\Delta_C$ has coefficients in $D(I)$. Since $D(I)$ is a DCC set \cite[Lemma 4.4]{mp04} and $\Diff_C\le \Delta_C$, 
 Lemma \ref{l_acc.num} implies that  
there are only finitely many possibilities for the coefficients of $\Diff_C$ at $q$, 
and the existence of $m_1$ follows from Lemma \ref{l_diff}.

Now we show that there exists $m_2 \in \mathbb Z_{>0}$, depending only on $I$, such that 
$m_2(K_X+\Delta)$ is Cartier around  any closed point $q \in X$ such that $(X, \Delta)$ is klt at $q$. 
Let $f\colon Y\to X$ be the minimal resolution at $q$. We have
\[
K_Y+f_*^{-1}\Delta+\sum_{i=1}^n(1-a_i)E_i
=f^*(K_X+\Delta)\]
where $E_1,\dots,E_n$ are the $f$-exceptional prime divisors, ordered so that
\[
a_1 \leq a_2 \leq \cdots \leq a_n.
\]
By Lemma~\ref{l_klt_index}, if $\ell$ is a positive integer such that 
\[
\ell a_1, \ell a_2, \ldots, \ell a_8 \in \mathbb Z
\]
and such that $\ell \Delta$ is a $\mathbb Z$-divisor, then $2\ell (K_X+\Delta)$ is Cartier around $q$. 
Thus, it is enough to find $n_0 \in \mathbb Z_{>0}$, depending only on $I$,   such that 
$n_0a_1, \dots, n_0a_8\in \mathbb Z$. 

To this end, we first extract a prime divisor $E_1$ such that  $a(E_1, X, \Delta)=a_1$ and we obtain a  birational morphism 
$g\colon Z \to X$
such that 
\[ K_Z+g_*^{-1}\Delta+(1-a_1)E_1=g^*(K_X+\Delta)\equiv 0.
\]
By ACC for the minimal log discrepancy \cite[Theorem 3.2]{alexeev93}, 
the set  $\{a_1\}_{(X, \Delta), x \in X}$ is an ACC set, 
hence $\{1-a_1\}$ is a DCC set. 
By Lemma \ref{l_acc.num}, there are finitely many possibilities for $a_1$. 
Then, we extract $E_2$ with $a(E_2, X, \Delta)=a_2$, and apply the same argument, to show that there are only finitely many possibilities for $a_2$. 
Repeating the same argument eight times, 
we see that there are finitely many possibilities for $a_1, \ldots, a_8$.
Thus, we may find a  positive integer $n_0$ as above and the claim follows. 
\end{proof}

The following result is well known at least in characteristic zero  \cite[Lemma 8.3.15]{prokhorov01}. 
We include the proof for the sake of completeness.

\begin{lemma} \label{l_invar}
Let $(Z, C+B)$ be a two dimensional $\mathbb Q$-factorial log canonical pair, 
where $C$ is a prime divisor and $B$ is an effective $\mathbb Q$-divisor.
Let $\pi \colon Z \to T$ be a projective morphism onto a smooth curve $T$ such that $\pi_*\mathcal O_Z=\mathcal O_T$. 
Assume that 
\begin{itemize}
\item{$C$ is $\pi$-horizontal,}
\item{$K_Z+C+B \equiv_{\pi} 0$, and}
\item{$\rho(Z/T)=1$.}
\end{itemize}
Then the following  hold: 
\begin{enumerate}
\item{For every closed point $t \in T$, $\pi^{-1}(t)_{\mathrm{red}}$ is isomorphic to $\mathbb P^1$.}
\item{$[K(C):K(T)] \leq 2$.}
\item{Assume that $C$ is normal, and let us  define 
 $K_C+B_C=(K_Z+C+B)|_C$ by adjunction. 
If $[K(C):K(T)] = 2$ and the field extension $K(C)/K(T)$ is separable, 
then $B_C$ is invariant under the action of the Galois group $\mathrm{Gal}(K(C)/K(T))$.}
\end{enumerate}
\end{lemma}

\begin{proof}
Let $F:=\pi^{-1}(t)_{\mathrm{red}}$. 
Note that $F$ is irreducible, because $Z$ is $\mathbb Q$-factorial and $\rho(Z/T)=1$. 
Since $C$ is $\pi$-horizontal, we obtain 
\[
(K_Z+F) \cdot F=K_Z \cdot F<(K_Z+C+B) \cdot F=0.
\]
Thus,  \cite[Theorem 3.19(1)]{tanaka12} implies (1). 

Let $G$ be a general fibre of $\pi$. 
Then  $G$ is integral \cite[Corollary 7.3]{badescu01}. Thus, $G \simeq \mathbb P^1$ by (1). 
In particular, it follows that 
\[
(K_Z+B) \cdot G \geq K_Z \cdot G=(K_Z+G) \cdot G=\deg(K_G+\Diff_G) \geq -2,
\]
which implies $C \cdot G=-(K_Z+B) \cdot G \leq 2.$ 
Thus, (2) holds. 

We now prove (3). 
Note  that $B$ is $\pi$-vertical, since  
\[
0=(K_Z+C+B) \cdot G \geq -2+2+B \cdot G=B \cdot G,
\]
where $G$ is a general fibre of $\pi$. 
Since $\rho(Z/T)=1$, we obtain $K_Z+C \equiv_{\pi} 0$. 
Since the divisor $\pi^*(t)|_C$ is $\mathrm{Gal}(K(C)/K(T))$-invariant for every closed point $t \in T$, 
so is $B|_C$.  In particular, we may assume $B=0$ and $B_C=\Diff_C$. 

Let $Q \in T$ be such that $\pi^{-1}(Q) \cap C$ consists of two distinct points $Q_1, Q_2$. 
It is enough to show that the coefficients of $\Diff_C$ at the points $Q_1$ and $Q_2$ coincide. 
We may write $\pi^*Q=mF$ where $F$ is a prime divisor and $m$ is a positive integer. 
By (1), $F \simeq \mathbb P^1$. 

We show that 
the pair $(Z, C+F)$ is log canonical. 
We have
\[
0\sim_{\mathbb Q}(K_Z+C+F)|_F=K_F+{\Diff}_F+C|_F.
\]
Since $(Z,C+F)$ is not plt at the points $Q_1$ and $Q_2$, inversion of adjunction implies that the pair $(F,{\Diff}_F+C|_F)$ is not klt at the points $Q_1$ and $Q_2$. 
Thus, 
\[
\Diff_F+C|_F\ge Q_1+Q_2,
\]
and since $\deg (\Diff_F+C|_F)=\deg (-K_F)=2$, equality holds. Thus,  inversion of adjunction implies that $(Z,C+F)$ is log canonical. 

Since $\pi|_C\colon C \to T$ is \'etale over $Q$, 
we have $mF|_C=\pi^*Q|_C=Q_1+Q_2$ and 
\[
(K_Z+C+F)|_C=K_C+\Diff_C +\frac 1 m (Q_1+Q_2).
\]
In particular, by inversion of adjunction again, the coefficients of $\Diff_C$ at $Q_1$ and $Q_2$ are equal to $1-\frac{1}{m}$, and the claim follows. 
\end{proof}


\subsection{Global F-adjunction} \label{subsection_global-f-adjunction}

We now summarise some known results on $F$-adjunction which we will use in this subsection. 
Let $X$ be a normal variety, $S$ a prime divisor and $B$ an effective $\mathbb Q$-divisor on $X$ 
such that $S \not\subseteq \Supp B$. 
Let $\nu\colon S' \to S$ be the normalisation. 
Assume that $(p^e-1)(K_X+S+B)$ is Cartier for some $e \in \mathbb Z_{>0}$ (this is equivalent to saying that the Cartier index of $K_X+S+B$ is not divisible by $p$).
By adjunction, we may write 
\[
\nu^*((K_X+S+B)|_S)=K_{S'}+B_{S'}
\]
 (cf.\ Subsection~\ref{s_diff}). 
There are natural homomorphisms \cite[discussion below Proposition 2.5]{schwede14}
\[
\phi^X_{S+B}:F_*^e\mathcal O_X(-(p^e-1)(K_X+S+B)) \to \mathcal O_X,
\]
and
\[
\phi^{S'}_{B_{S'}}:F_*^e\mathcal O_{S'}(-(p^e-1)(K_{S'}+B_{S'})) \to \mathcal O_{S'}.
\]
Let 
\[
\mathcal L:=\mathcal O_X(-(p^e-1)(K_X+S+B)).
\]
Since $S$ is an $F$-pure centre of $(X, S+B)$ [ibid, Definition 5.1], 
we obtain a commutative diagram [ibid, Remark 5.2]: 
\[
\begin{CD}
0 @>>> F^e_*(\mathcal L\otimes_{\mathcal O_X} \mathcal O_X(-S)) @>>> F^e_*\mathcal L @>>> F^e_*(\mathcal L|_S) @>>> 0\\
@. @VVV @VV\phi^X_{S+B} V @VV\psi V\\
0 @>>> \mathcal O_X(-S) @>>> \mathcal O_X @>>> \mathcal O_S @>>> 0.
\end{CD}
\]
We also obtain the following commutative diagram \cite[Lemma 8.1]{schwede09}: 
\[
\begin{CD}
F^e_*(\mathcal L|_S) @>>> F^e_*(\nu^*(\mathcal L|_S))\\
@VV\psi V @VV\phi^{S'}_{B_{S'}} V\\
\mathcal O_S @>>> \mathcal O_{S'}.
\end{CD}
\]
Note that the right vertical arrow coincides with $\phi^{S'}_{B_{S'}}$ by \cite[Theorem 5.3]{das15}.

\medskip

Under some additional assumptions, we will show that if $-(K_X + S + B)$ is nef and $(S',B_{S'})$ is globally $F$-split, then  $(X,S+B)$ is globally $F$-split as well. 
The following three lemmas correspond to the  cases  
\[
\kappa(X, -(K_X+S+B))=2,1\text{ and }0
\]
respectively. 

\begin{lemma} \label{l_gladj}Let $(X,S+B)$ be a projective log pair 
where $S$ is a prime divisor and $B$ is an effective $\mathbb Q$-divisor such that 
$S \not\subseteq \Supp B$. Assume that 
\begin{enumerate}
\item  $-(K_X + S + B)$ is ample,
	\item $S$ is normal and if $K_S+B_S=(K_X+S+B)|_S$ is defined by adjunction, 
then $(S, B_S)$ is globally $F$-split, and
	\item $(p^e-1)(K_X + S + B)$ is Cartier for some positive integer $e$.
\end{enumerate}
Then $(X,S+B)$ is globally $F$-split.
\end{lemma}

\begin{proof}

Let $\mathcal{L} = \mathcal O_X( -(p^e-1)(K_X + S + B))$. We have the following diagram:
\[
\begin{CD}
H^0(X, F_*^e\mathcal L) @>>> H^0(S, F_*^e\mathcal L|_{S}) @>>> H^1(X, F^e_*(\mathcal{L}(-S)))\\
@VV \phi^X_{S+B} V @VV \phi^S_{B_S} V\\
H^0(X, \mathcal{O}_X) @> \simeq >> H^0(S, \mathcal{O}_{S}).
\end{CD}
\]
 (2) implies that the right vertical arrow $\phi^S_{B_S}$ is surjective. 
 By   Serre vanishing, after possibly replacing $e$ by a larger multiple,  (1) implies that $H^1(X, F^e_*(\mathcal{L}(-S))) = 0$. By a diagram chase, it follows that $\phi^X_{S+B}$ is surjective. 
Thus,  $(X,S+B)$ is globally $F$-split. 
\end{proof}

\begin{lemma}\label{lem-MFS}
Assume that $\mathrm{char}\,k>2$. 
Let $(Z,C+B)$ be a two dimensional projective $\mathbb Q$-factorial log canonical pair, 
where $C$ is a prime divisor and $B$ is an effective $\mathbb Q$-divisor.  Let $\pi\colon Z\to T$ be a morphism onto a smooth projective curve $T$ such that  $\pi_*\mathcal{O}_Z=\mathcal{O}_T$. Assume that 
\begin{enumerate}
\item $C$ is $\pi$-horizontal, 
\item $-(K_Z+C+B) \equiv \pi^*A$ for some ample $\mathbb{Q}$-divisor $A$ on $T$,
\item $C$ is normal and if $K_C+B_C=(K_Z+C+B)|_C$ is defined by adjunction, 
then $(C, B_C)$ is globally $F$-split, 
\item $(p^e-1)(K_Z + C + B)$ is Cartier for some positive integer $e$, and
\item $\rho(Z/T)=1$.
\end{enumerate}
Then $(Z,C+B)$ is globally $F$-split.  
\end{lemma}

\begin{proof}
Let 
$\mathcal{L} := \mathcal O_Z( -(p^e-1)(K_Z + C + B))$. We have the following commutative diagram:
\[
\begin{CD}
H^0(Z, F_*^e\mathcal L) @>>> H^0(C, F_*^e\mathcal L|_{C})\\
@VV \phi^Z_{C+B} V  @VV \phi_{B_C}^C V\\
H^0(Z, \mathcal{O}_Z) @>\simeq >> H^0(C, \mathcal{O}_{C}).
\end{CD}
\]
By (3),  the  right vertical arrow  $\phi_{B_C}^C$ is surjective. 

By (2) and (4), we can find an ample Cartier divisor $A'$ on $T$ such that 
$\mathcal L  \simeq \mathcal O_Z(\pi^*A')$ \cite[Theorem 0.4]{tanaka12a}. 
Since $\pi_*\mathcal O_Z=\mathcal O_T$, we have that
\[
H^0(Z, \mathcal{L}) \simeq H^0(T, \mathcal{O}_{T}(A')),
\]
and so the image of the upper horizontal arrow is exactly the $G$-invariant part: $H^0(C, F^e_* \mathcal{L}
|_C)^G$, where $G$ is the Galois group of $K(C)/K(T)$. 

Let $f \in H^0(C, F^e_*\mathcal{L}|_C)$ be such that $\phi_{B_C}^C(f)=1$. 
Since $B_C$ is $G$-invariant by Lemma~\ref{l_invar}, 
we have that $\phi_{B_C}^C$ is $G$-equivariant. In particular,
\[
\phi_{B_C}^C\Big(\frac{1}{|G|}\sum_{g \in G}g(f) \Big) = 1.
\]
We can divide by $|G|$, because $\mathrm{char}\,k>2$ and $|G| \leq 2$ by Lemma~\ref{l_invar}. 
Thus, $\phi^Z_{C+B}$ is surjective and $(Z,C+B)$ is globally $F$-split.  
\end{proof}

\begin{lemma}\label{l_pic1-non-plt}
Let $(Z, C+B)$ be an $n$-dimensional projective log canonical pair with $n \leq 2$, 
where $C$ is a prime divisor and $B$ is an effective $\mathbb Q$-divisor. Assume that  
\begin{enumerate}
\item $(p^e-1)(K_Z + C + B) \sim 0$ for some positive integer $e$, 
\item $H^1(Z,\mathcal{O}_Z)=0$, and
\item if $n=2$, then $(Z,C+B)$ is not plt. 
\end{enumerate}
Then $(Z,C+B)$ is globally $F$-split. 
\end{lemma}

\begin{proof}
We only prove the case $n=2$, as the case $n=1$ is easier. 

Let $\nu\colon C'\to C$ be the normalisation.
We may write 
\[
\nu^*((K_Z+C+B)|_C)=K_{C'}+B_{C'}
\]
for some effective $\mathbb Q$-divisor $B_{C'}$ on $C'$ (cf.\ Subsection \ref{s_diff}).
By (3) and by inversion of adjunction, it follows that $\llcorner B_{C'} \lrcorner \neq 0$. 
Thus, by (1), it follows that $C'$ is isomorphic to $\mathbb P^1$. 
Therefore, by the same Lemma in dimension one, $(C', B_{C'})$ is globally $F$-split. 

Let $e$ be a positive integer satisfying (1) and let 
\[
\mathcal L:=\mathcal O_Z(-(p^{e}-1)(K_Z+C+B)).
\]
By (1), we have that $\mathcal L\simeq \mathcal O_Z$. 
We consider the commutative diagram
\[
\begin{CD}
H^0(Z, F_*^e\mathcal L) @>>> H^0(C, F_*^e(\mathcal L|_{C})) @>>> H^1(Z, F_*^e(\mathcal L(-C)))\\
@VV\phi^Z_{C+B} V @VV\psi V\\
H^0(Z, \mathcal{O}_Z) @> \simeq >> H^0(C, \mathcal{O}_{C}).
\end{CD}
\]

\medskip

We now show that $\psi$ is surjective. 
We have the following commutative diagram: 
\[
\begin{CD}
H^0(C, F_*^e(\mathcal L|_{C})) @>\alpha >> H^0(C', F_*^e(\nu^*(\mathcal L|_{C})))\\
@VV\psi V @VV\phi_{B_{C'}}^{C'} V\\
H^0(C, \mathcal{O}_C) @> \simeq >> H^0(C', \mathcal{O}_{C'}).
\end{CD}
\]
Since $\alpha$ is induced from the natural homomorphism $\mathcal O_C \to \nu_*\mathcal O_{C'}$ and $\mathcal L \simeq \mathcal O_Z$, it follows that $\alpha$ is an isomorphism. 
Since $(C', B_{C'})$ is globally $F$-split, $\phi_{B_{C'}}^{C'}$ is surjective. 
Thus, $\psi$ is  surjective as well. 

\medskip

Since $\mathcal L\simeq \mathcal O_Z$, (2) implies 
\[
\begin{aligned}
H^1(Z, F_*^e(\mathcal L(-C)))&\simeq H^1(Z, \mathcal L(-C))\\
&\simeq  H^1(Z, \mathcal O_Z(-C)) \xhookrightarrow{\hphantom{Tr}} H^1(Z, \mathcal O_Z)=0.
\end{aligned}
\]
Since $\psi$ is  surjective, by a diagram chase  it follows that 
$\phi^Z_{C+B}$ is surjective as well, and, in particular,  $(Z,C+B)$ is globally $F$-split. Thus, the claim follows.
\end{proof}

\subsection{Perturbations}
Let $(X,S+B)$ be a log  pair such that $S$ is a prime divisor and $B$ is an effective $\mathbb Q$-divisor 
such that $S \not\subseteq \Supp B$. If the Cartier index of $K_X+S+B$ is divisible by $p$, then, in order to apply the results from the previous subsection, we need to perturb the coefficients of $B$. 

We begin with the following well known result: 

\begin{lemma}
\label{l_basic-perturbation} Let $(X,\Delta)$ be a quasi-projective
 globally sharply $F$-split pair. Then there exists an effective $\mathbb{Q}$-divisor $H$ such that  $(X,\Delta+H)$ is globally sharply $F$-split and 
\[
(p^e-1)(K_X+\Delta+H) \sim 0
\]
for some positive integer $e$.
\end{lemma}

\begin{proof}
This follows from the same proof as \cite[Theorem 4.3(ii)]{schwedesmith10}.
\end{proof}

\begin{lemma} \label{l_perturbation} Let $(X,\Delta)$ be a two-dimensional quasi-projective sharply $F$-pure log pair. Then there exists an effective $\mathbb{Q}$-divisor $H$ such that $(X,\Delta+H)$ is sharply $F$-pure, and 
\[
(p^e-1)(K_X + \Delta+H)
\]
is Cartier for some positive integer $e$.
\end{lemma}

\begin{proof} 
Let  $U\subseteq X$ be an affine open subset such that
 none of the irreducible components of $\Delta$ are contained in $X \setminus U$ 
and such   that, around each point of  $X \setminus U$, we have that $X$ is smooth, the support of $\Delta$ is simple normal crossing and $(X, \Delta)$ is plt.  
Since $(X,\Delta)$ is sharply $F$-pure and $U$ is affine, it follows that   $(U, \Delta|_U)$ is globally sharply $F$-split.

By Lemma \ref{l_basic-perturbation}, there exists an effective $\mathbb{Q}$-divisor $H'$ on $X$ such that $(U, (\Delta+H')|_U)$ is globally sharply $F$-split and
\[
(p^e-1)(K_U + (\Delta+ H')|_U) \sim 0, 
\]
 for some positive integer  $e$. Since $X$ is smooth around each point of  $X\setminus U$, it follows that 
 $(p^e-1)(K_X+\Delta+H')$ is Cartier on $X$. 
For every positive integer $d$, we have  
\[
(K_X+\Delta+H')-(K_X+\Delta+\frac{1}{p^d+1}H')=\frac{p^d}{p^d+1}H'.
\]
Thus, for every sufficiently large positive integer $d$,  there exists $e(d) \in \mathbb Z_{>0}$ such that 
\[
(p^{e(d)}-1)(K_X+\Delta+\frac{1}{p^d+1}H')
\]
is Cartier. 

Let $H:=\frac{1}{p^d+1}H'$. It is enough to show that, after possibly replacing $d$ by a larger value, the pair $(X,\Delta+H)$ is sharply $F$-pure around any point $q\in X\setminus U$. By assumption, around $q$ we have that  $X$ is smooth,  the support of $\Delta$ is simple normal crossing and $(X,\Delta)$ is plt.
Thus, if $q$ is not contained in the support of $\llcorner \Delta\lrcorner$ and $d$ is sufficiently large, then $(X,\Delta+H)$ is strongly $F$-regular.  On the other hand, if $q$ is contained in the support of $\llcorner \Delta\lrcorner$ then the claim follows from inversion of adjunction (e.g.\ see \cite[Main Theorem]{schwede09}, \cite[Theorem A]{das15}).
\end{proof}

\begin{lemma} \label{l_mfs-perturbation} 
Assume that $\mathrm{char}\,k>2$. 
Let $(Z,C+B)$ be a two dimensional projective sharply $F$-pure pair, 
where $C$ is a prime divisor and $B$ is an effective $\mathbb Q$-divisor. Let $\pi\colon Z\to T$ be a morphism onto a smooth projective curve $T$ such that  $\pi_*\mathcal{O}_Z=\mathcal{O}_T$. Assume that 
\begin{enumerate}
\item $(Z,B)$ is klt, 
\item $-(K_Z+C+B) \equiv \pi^*A$ for some ample $\mathbb Q$-divisor $A$ on $T$, 
\item $C$ is $\pi$-horizontal, 
\item $\rho(Z/T)=1$, and
\item $(p^d-1)B$ is a $\Z$-divisor for some positive integer $d$.
\end{enumerate}
Then there exists an effective $\mathbb Q$-divisor $E$, 
whose support is contained in some fibres of $\pi$, such that $(Z, C+B+E)$ is log canonical and
\[
(p^e-1)(K_Z + C +B+ E)
\] 
is Cartier for some positive integer $e$.
\end{lemma}

\begin{proof}
By Lemma \ref{l_invar}, there exists a non-empty open subset $T^0 \subseteq T$ 
such that the induced morphism 
$\pi^{-1}(T^0) \to T^0$ is a $\mathbb P^1$-bundle. 
By (2) and (5), 
it follows that, after possibly shrinking $T^0$, there exists a positive integer $d$ such that 
\[
(p^d-1)(K_Z+C+B)|_{\pi^{-1}(T^0)} \sim 0.
\] 
By (1) of  Lemma~\ref{l_invar},  every fibre of $\pi$ is irreducible. Thus, 
we may write 
\[
(p^d-1)(K_Z + C + B) \sim \sum_{i=1}^r a_i F_i+\sum_{j=1}^s b_j F'_j,
\] 
where $a_i, b_j \in \Z$, $F_i=\pi^{-1}(t_i)_{\mathrm{red}}$, and $F'_j=\pi^{-1}(t'_j)_{\mathrm{red}}$, with 
\[
t_1,\dots,t_r,t'_1,\dots,t'_s\in T,
\] so that 
 $(Z, C+B)$ is plt (resp.\ not plt) along $F_i$ (resp.\ at some point of $F'_j$). 
Since $\llcorner C+B\lrcorner=C$ is $\pi$-horizontal and $(Z,B)$ is klt, by inversion of adjunction 
we can find $\alpha_i \in \mathbb Q_{>0}$ for $i=1,\dots, r$ such that 
$(Z, C+B+\sum_{i=1}^r \alpha_i F_i)$ is log canonical and 
$(p^{e_1}-1)(K_Z+C+B+\sum_{i=1}^r \alpha_i F_i)$ is Cartier around $\bigcup_{i=1}^r F_i$ 
for some $e_1 \in \mathbb Z_{>0}$. 

Fix an index $1\leq j \leq s$. 
It is enough to find $e_2 \in \mathbb Z_{>0}$ such that $(p^{e_2}-1)b_jF'_j$ is Cartier.
By construction, there exists a zero-dimensional log canonical centre $z \in F'_j$ of $(Z,C+B)$.  Since $(Z,B)$ is klt, it follows that $z\in C$. 
Let $m_z$ be the Cartier index of $F'_j$ at $z$. 
Since $z \in C$, it follows  that $m_z F'_j \cdot C \geq 1$. 
We can write $\pi^*(\pi(z))=mF'_j$ for some positive integer $m$. 
 Lemma~\ref{l_invar} implies that $mF'_j \cdot C=\pi^*(\pi(z)) \cdot C \le 2$.  Thus,  
\[
\frac{1}{m_z}\le F'_j \cdot C \le \frac{2}{m} ,
\]
and in particular $2m_z \geq m$. Since $m_z$ divides $m$, it follows that either  $m=m_z$ or $m=2m_z$.

By applying Lemma \ref{l_basic-perturbation} 
to an affine open neighbourhood $U$ of $z$,
we can find $e_3 \in \mathbb Z_{>0}$ and an effective $\Q$-divisor $H$ on $U$ 
such that $(U, (C+B)|_{U}+H)$ is sharply $F$-pure and 
\[
(p^{e_3}-1)(K_{U}+(C+B)|_{U}+H) \sim 0.
\] 
Since $z \in Z$ is a zero-dimensional log canonical centre of $(Z, C+B)$, it follows that  $z$ is not contained in the support of $H$.
Therefore, $(p^{e_3}-1)(K_Z+C+B)$ is Cartier around $z$. 
Since 
\[
(p^d-1)(K_Z + C + B) \sim \sum_{i=1}^r a_i F_i+\sum_{j=1}^s b_j F'_j,
\] 
it follows that $(p^{e_4}-1)b_jF'_j$ is Cartier around $z$ for some $e_4 \in \mathbb Z_{>0}$. Since $m\in \{m_z,2m_z\}$, we have that  $2(p^{e_4}-1)b_jF'_j$ is Cartier. 
Since $\mathrm{char}\,k \neq 2$, we can find $e_2\in \mathbb Z_{>0}$ such that $(p^{e_2}-1)b_jF'_j$ is Cartier. 
Thus, the claim follows. \end{proof}

\subsection{Flat families of log pairs and liftability}

\begin{definition} Given a reduced noetherian separated scheme $T$, 
\emph{a flat family of log pairs} $(\mathcal{X},\mathcal{B})$  over $T$
consists of the data of
\begin{itemize}
\item a normal scheme $\mathcal X$ and an effective $\mathbb Q$-divisor $\mathcal B$ on $\mathcal X$ such that $K_{\mathcal X}+\mathcal B$ is $\mathbb Q$-Cartier, 
\item a separated and flat morphism of finite type $\pi\colon \mathcal X \to T$ 
such that all the fibres of $\pi$ are geometrically normal, and 
\item for each irreducible component $\mathcal{B}_i$ of $\mathcal{B}$, 
the induced morphism $\pi|_{\mathcal B_i}\colon \mathcal{B}_i \to T$ is flat and all the fibres of $\pi|_{\mathcal B_i}$ are geometrically integral. 
\end{itemize}  
\end{definition}

\begin{definition}
Let $(\mathcal{X},\mathcal{B})$ be a flat family of log pairs over a reduced noetherian separated scheme $T$ and let $(X,B)$ be a log pair  over 
an algebraically closed field $k$. We say that $(X,B)$ is a \emph{geometric fibre} of  $(\mathcal{X},\mathcal{B})$, if there exists a cartesian diagram
\[
\begin{CD}
X @>\beta>> \mathcal X\\
@VVV @VVV\\
\Spec k @>\alpha >> T 
\end{CD}
\]
such that $\beta^*(\mathcal{B}) = B$, 
where if we write $\mathcal B=\sum_{i=1}^r b_i \mathcal B_i$ with $\mathcal B_1,\dots,\mathcal B_r$ prime components,  then
we define 
\[
\beta^*(\mathcal{B}):=\sum_{i=1}^r b_i\beta^{-1}(\mathcal B_i).
\] 
Note  that, for each $i=1,\dots,r$,  the scheme-theoretic inverse image $\beta^{-1}(\mathcal B_i)$ is a prime divisor on $X$. 
\end{definition}

\begin{definition}\label{d-liftable}
Let $X$ be a smooth variety over a perfect field $k$ of characteristic $p>0$, 
and let $D$ be a simple normal crossing divisor on $X$. Write $D = \sum_{i=1}^r D_i$, where $D_i$ are the irreducible components of $D$. We say that the pair $(X,D)$ {\em lifts to characteristic zero over a smooth base} if 
there exist 
\begin{itemize}
	\item a scheme $T$ smooth and separated over $\Spec \Z$,
	\item a smooth and separated morphism $\mathcal{X} \to T$,
	\item effective Cartier divisors $\mathcal{D}_1, \dots, \mathcal D_r$ on $\mathcal{X}$ 
such that the scheme-theoretic intersection $\bigcap_{i \in J} \mathcal{D}_i$ for any subset $J \subseteq \{1, \ldots, r\}$ is smooth over $T$, and
	\item a morphism $\alpha \colon \Spec k \to T$,
\end{itemize}
such that the base changes of the schemes $\mathcal{X}, \mathcal{D}_1, \ldots, \mathcal{D}_r$ over $T$ by $\alpha \colon \Spec k \to T$ are isomorphic to $X, D_1, \ldots, D_r$, respectively. 
\end{definition}

We refer to \cite[Definition 8.11]{EV92} for the definition of  liftability to the second Witt vectors $W_2(k)$.

\begin{remark}
Under the same assumption as in Definition~\ref{d-liftable}, 
if $(X,D)$ lifts to characteristic zero over a smooth base, then $(X,D)$ also lifts to $W_2(k)$. 
Indeed, since $W_2(k)$ is a henselian local ring \cite[Proposition 2.8.4]{Fu15}, and $T$ is smooth over $\Spec \Z$, 
the morphism $\Spec k \to T$ lifts to a morphism $\Spec W_2(k) \to T$ by \cite[Proposition 2.8.13]{Fu15}:
\[
\begin{CD}
\mathcal{X} \times_{T} W_2(k) @>>> \mathcal{X}\\
@VVV @VVV\\
\Spec W_2(k) @>>> T.
\end{CD}
\]
\end{remark}

\section{Proof of the main theorem for $\epsilon$-klt log del Pezzo} \label{section-bounded-case}

The goal of this section is to prove the main theorem for $\epsilon$-klt log del Pezzo pairs, for any fixed $\epsilon>0$ (Proposition~\ref{prop-bounded-case}). 
The idea of the proof is that the family of $\epsilon$-klt log del Pezzos over $\Spec \Z$ is bounded (Lemma~\ref{lem-bounded-family}). 
The same result is known to hold  over a fixed algebraically closed field 
 \cite[Theorem 6.9]{alexeev94}. 
Our argument follows the same methods, however we include the proof for  completeness.  

Let $S$ be a noetherian separated scheme,  let  
$Y$ be a scheme which is projective and flat over $S$ and let 
$H$ be an invertible sheaf on $Y$, which is ample over $S$. Let $\phi \in \mathbb{Q}[t]$ be an arbitrary polynomial. 
We define  $\Hilb^{\phi,\, H}_{Y/S}$ to be the Hibert functor such that, for any scheme $T$ over $S$,  $\Hilb^{\phi,\, H}_{Y/S}(T)$ is the set of  closed subschemes in $Y\times_S T$ which are flat over $T$, and with Hilbert polynomial equal to $\phi$ with respect to the pull-back of $H$ on $Y\times_S T$ (see  \cite[Section 5.1]{fga2005} for more details).  
We also define $\Div^{\phi,\, H}_{Y/S}$ to be the functor such that, for any scheme $T$ over $S$, 
$\Div^{\phi,\, H}_{Y/S}(T)$ is the set of  effective Cartier divisors on $Y\times_S T$ which are flat over $T$, and with Hilbert polynomial equal to $\phi$ with respect to the pull-back of $H$ on $Y\times_S T$ (see [ibid, Section 9.3]). 


Since $\Hilb^{\phi,\, H}_{Y/S}$ is representable by a projective scheme over $S$ [ibid, Theorem 5.14 and Subsection 5.1.3], it follows that $\Div^{\phi,\, H}_{Y/S}$ is representable by a quasi-projective scheme over $S$ [ibid, Theorem 9.3.7].


\begin{lemma}  \label{lem-bounded-family} 
Let $I \subseteq (0,1) \cap \mathbb Q$ be a finite set and let $\epsilon>0$. 

Then  there exists a flat family of log pairs $(\mathcal{X},\mathcal{B})$ over a reduced quasi-projective scheme $T$ over $\Spec \Z$, such that every $\epsilon$-klt log del Pezzo pair $(X,B)$  over any algebraically closed field $k$ 
of characteristic $p>5$,  with the  coefficients of $B$ contained in $I$, is a geometric fibre of $(\mathcal{X},\mathcal{B})$.
\end{lemma}

\begin{proof}
There exist positive integers $b$ and $n$ which satisfy the following property 
(e.g.\ see \cite[Corollary 1.4 and Remark 6.3]{witaszek15}): 
for every $(X, B)$ over $k$ as in the Lemma, we can find a very ample divisor $H$ on $X$ 
such that 
\begin{itemize}
\item $|H| \text{ embeds } X \text{ into } \mathbb P_k^b, $
\item $ nB$ is Cartier, 
\item $|\phi_{X,i}| \leq b \text{ for all } i \geq 0$ where $\phi_X =\sum_{i \geq 0}\phi_{X,i} t^i\in \Q[t]$ is the Hilbert polynomial of $X$ with respect to $H$, and 
\item  $|\psi_{nB, i}| \leq b \text{ for all } i \geq 0$ where $\psi_{nB} =\psi_{nB,1} t+\psi_{nB, 0} \in \Z[t]$ is the Hilbert polynomial of $nB$ with respect to $H$.
\end{itemize}
We may assume that $I=(0, 1) \cap \frac{1}{n}\Z$. Further,
since $\dim X = 2$, we have that $2\phi_X \in \Z[t]$ and $\deg \phi_X=2$. 

In particular, if $G$ is the functor defined by 
\[
G:=\bigcup_{\phi} \Hilb^{\phi,\, \mathcal{O}_{\mathbb{P}_{\Z}^b/ \Z}(1)}_{\mathbb{P}_{\Z}^{b}/ \Z}
\]
where the union is taken over all $\phi=\sum\phi_i t^i\in \Q[t]$  such that $2\phi \in \Z[t]$, $\deg \phi = 2$, and $|\phi_i| \leq b$ for all $i\geq 0$, then $X\in G(\Spec k)$. 

By \cite[Theorem 5.14 and Subsection 5.1.3]{fga2005}, the functor $G$ is representable by a projective scheme $S$ over $\Spec \mathbb Z$. Let $\mathcal U\subseteq \mathbb P^b_{\mathbb Z}\times_{\mathbb Z} S$ be its universal closed subscheme. 
Let $\mathcal{H}$ be the pull-back of  $\mathcal{O}_{\mathbb{P}_{\Z}^b/ \Z}(1)$ to $\mathcal U$.  Then $\mathcal H$ is ample over $S$. We have
\[
(nB \subseteq X) \in \bigcup_{\psi} \Div^{\psi,\, \mathcal{H}}_{\mathcal U/S}(\Spec k), 
\]
where the union is taken over all $\psi=\psi_1 t+\psi_0 \in \Z[t]$ 
such that $|\psi_0| \leq b$ and $|\psi_1| \leq b$. 

Let $T$ be the reduction of the scheme representing $\bigcup_{\psi} \Div^{\psi,\, \mathcal{H}}_{\mathcal U/S}$. 
Note that $T$ is quasi-projective over $\Spec\,\Z$. 
Let $\mathcal B'$ be the universal effective Cartier divisor on $\mathcal X:=\mathcal U \times_S T$. 
Let $\mathcal B:=\frac{1}{n}\mathcal B'$. 
Take a generic point $\zeta$ of $T$. 
By taking a base change of some finite morphism of an open subset of
$T$, we may assume that the fibres of the irreducible components of $\mathcal{B}$ over $\zeta$ are geometrically integral. Thus, there exists an open subscheme $U \subseteq T$, such that $(\mathcal{X}|_U, \mathcal{B}|_U)$ is a flat family of log pairs. By replacing $T$ by $U \amalg T \backslash U$ and repeating the same argument to $T \backslash U$, we may conclude the proof.
\end{proof}

We now show the main result of this Section. 

\begin{proposition} \label{prop-bounded-case} Let $I \subseteq (0,1) \cap \mathbb Q$ be a finite set and let $\epsilon>0$. Then, there exists a positive integer $p(I, \epsilon)$ which satisfies the following property:

Let $(X,B)$ be an $\epsilon$-klt log del Pezzo pair  over an algebraically closed field of characteristic $p > p(I, \epsilon)$ such that the coefficients of $B$ are contained in $I$. 
Then there exists a log resolution $\mu \colon V \to X$ of $(X, B)$ 
such that $(V, \Exc(\mu) \cup \mu^{-1}_*(\Supp B))$ lifts to characteristic zero over a smooth base.
\end{proposition}

\begin{proof}
We fix a positive integer $p(I,\epsilon)>5$. We will replace $p(I,\epsilon)$ by a larger number if necessary. 
By Lemma \ref{lem-bounded-family}, there exists a flat family of log pairs $(\mathcal{X}, \mathcal{B})$ over a reduced quasi-projective scheme $T$ over $\Spec \Z$, such that 
every pair $(X, B)$ as in the Proposition is a geometric fibre of 
the pair $(\mathcal{X}, \mathcal{B})$. 

\medskip
\textbf{Claim:} 
Let $T'$ be an irreducible component of $T$ 
such that the field $K(T')$ is of characteristic zero. 

Then there exists a 
dominant morphism $S \to T'$ 
from an integral scheme $S$ which satisfies the following properties: 
\begin{enumerate}
\item[(a)]{$S$ is smooth over $\Spec\,\Z$.}
\item[(b)]{If we set $\mathcal X_S:=\mathcal X \times_T S$, then 
there exists a projective birational morphism $\mu_S\colon\mathcal V_S \to \mathcal X_S$ over $S$ such that the induced morphism $\mu_{S, s} \colon \mathcal V_s \to \mathcal X_s$ between any fibres over $s \in S$ 
is birational, $\mathcal V_S$ is smooth over $S$, and  
\[
\Exc({\mu}_S) \cup ({\mu}_S)_*^{-1}(\Supp \mathcal B)
\] 
is simple normal crossing over $S$, i.e.\ all the strata are smooth over $S$. }
\end{enumerate}

\medskip

We first show the Proposition, assuming the Claim. By noetherian induction, we can find a surjective morphism 
$$\left(\coprod_{1\leq i \leq \alpha} S_i\right) \amalg \left(\coprod_{1 \leq j \leq \beta} S'_j\right) \to T$$  
such that all $S_i$ and $S'_j$ are separated integral schemes of finite type over $\Spec\, \Z$, 
any $K(S_i)$ (resp.\ $K(S'_j)$) is of characteristic zero (resp.\ $p_j>0$), 
and  each $S_i$ satisfies properties (a) and (b).
Then the Proposition follows after possibly increasing $p(I, \epsilon)$ 
so that $p(I,\epsilon)\ge p_j$ for all $j=1,\dots, \beta$. 

\medskip

We now show the Claim. 
We set $S:=T'$. 
It is enough to show that properties (a) and (b) hold 
after possibly replacing $S$ by a finite cover 
(i.e.\ a finite surjective morphism from an integral scheme) of an open subset. 

By our assumption, 
the field $K(S)=\mathcal O_{S, \xi}$ is of characteristic zero, 
where $\xi$ is the generic point of $S$. 
Thus, after replacing $S$ by an open subset, 
we may assume that $S$ is smooth over $\Spec\,\mathbb Z$, hence (a) holds. 

Let $(X_{\overline{\xi}}, B_{\overline{\xi}})$ be the base change to the algebraic closure $\overline{\xi}$ 
of the generic fibre $(X_{\xi}, B_{\xi})$ and take a log resolution ${\mu}_{\overline{\xi}}\colon V_{\overline{\xi}} \to X_{\overline{\xi}}$ of $(X_{\overline{\xi}}, B_{\overline{\xi}})$. 
After replacing $S$ by a finite cover of an open subset, 
we may assume that 
there exists a projective birational morphism over $S$ 
\[
{\mu}_S\colon \mathcal{V}_{S} \to \mathcal X_{S}
\]
whose base change to $\overline{\xi}$ is the same as 
$V_{\overline{\xi}} \to X_{\overline{\xi}}$, where $\mathcal{X}_{S}:=\mathcal{X} \times_T S.$ 

We check that property (b) holds, after replacing $S$ by an open subset.
Indeed, since ${\mu}_{\overline{\xi}}\colon V_{\overline{\xi}} \to X_{\overline{\xi}}$ is a projective birational 
morphism of surfaces, so are the morphisms 
$\mathcal{V}_{S, s} \to \mathcal X_{S, s}$ between the geometric fibres, for any geometric point $s$ in an open neighbourhood of $\xi$. 
Moreover, $V_{\overline{\xi}}$ and any stratum of  
$\Exc({\mu}_{\overline{\xi}}) \cup ({\mu}_{\overline{\xi}})_*^{-1}(\Supp \mathcal B_{\overline{\xi}})$ 
are smooth, hence  $\mathcal V_S$ and all the strata of 
$\Exc({\mu}_S) \cup ({\mu}_S)_*^{-1}(\Supp \mathcal B)$ are smooth.
\end{proof} 

\begin{remark}
By Proposition~\ref{prop-bounded-case}, it follows that $\epsilon$-klt log del Pezzo pairs in large characteristic satisfy (2) of Theorem~\ref{thm-main}. 
By  \cite[Theorem 1.2]{schwedesmith10} and  noetherian induction, 
it follows that (1) of Theorem~\ref{thm-main} holds 
for $\epsilon$-klt log del Pezzo pairs in large characteristic. 
However, we do not use this fact in this paper. 
\end{remark}

\section{Liftability to characteristic zero}

The goal of  this section is to study plt pairs $(Z, C_Z+B_Z)$, such that  $\llcorner C_Z+B_Z \lrcorner =C_Z $ is a prime divisor and $K_Z+C_Z+B_Z \equiv 0$. These pairs appear in the proof of Theorem~\ref{thm-main}, after running a suitable MMP, starting from some model over a log del Pezzo pair $(X,B)$ (cf.\ Section \ref{s_proof-main-thm}). 

We first show that, in large characteristic, such a pair $(Z, C_Z+B_Z)$ 
admits a log resolution  which lifts to characteristic zero (Proposition~\ref{p_log_resol_lift}). 
Then, in order to show that also the pair $(X, B)$ admits a log resolution which lifts to characteristic zero, 
we  study the behaviour of such a  liftability property under blow-ups (Lemma~\ref{l_blow_up_lift}).  

\medskip 

The following result  is a consequence of Lemma \ref{lem-bounded-family}:

\begin{lemma} \label{l-bounded-family2} Let $I \subseteq (0,1) \cap \mathbb Q$ be a DCC set. 

Then there exists a flat family of log pairs $(\mathcal{X},\mathcal{C}+\mathcal{B})$ over a reduced quasi-projective scheme $T$ over $\Spec \Z$, such that any two dimensional projective plt pair $(X,C+B)$, 
 over any algebraically closed field $k$ of characteristic $p>5$, satisfying
\begin{itemize}
	\item $\lfloor C + B \rfloor = C$, 
	\item the coefficients of $B$ are contained in $I$,
	\item $K_X + C +B \equiv 0$, and
	\item $C$ is ample,
\end{itemize}
is a geometric fibre of $(\mathcal{X},\mathcal{C} + \mathcal{B})$.
\end{lemma}

\begin{proof}
As in the proof of Proposition \ref{p_bounded-index}, there exists a positive integer $m$ depending only on $I$ such that 
\[
m(K_X+C+B) \sim 0\qquad\text{and}\qquad mC \text{ is Cartier}
\] 
for any pair $(X,C+B)$ satisfying the assumptions  in the Lemma.

In particular, $2m(K_X+\frac{1}{2}C+B)$ is Cartier, and if $\epsilon\in (0,\frac1{2m})$ then $(X, \frac 1 2 C+B)$ is $\epsilon$-klt. Thus,  Lemma \ref{lem-bounded-family} implies that there exists a flat family $(\mathcal X, \frac{1}{2}\mathcal C+\mathcal B)$ of log pairs such that $(X, \frac{1}{2}C+B)$ is a geometric fibre of
$(\mathcal X, \frac{1}{2}\mathcal C+\mathcal B)$. In particular,  $(\mathcal X, \mathcal C+\mathcal B)$ is the  required family. 
\end{proof}

\begin{proposition} \label{p_log_resol_lift}
Let $I \subseteq (0, 1) \cap \Q$ be a finite set. 

Then there exists a positive integer $p(I)$ which satisfies the following property:

For any two dimensional projective plt pair $(Z,C+B)$  over an algebraically closed field of characteristic $p>p(I)$ such that 
\begin{itemize}
\item $\llcorner C+B\lrcorner=C$, 
	\item $C$ is ample,
	\item $K_Z + C + B \equiv 0$,  and 
\item{the coefficients of $B$ are contained in $I$, }
\end{itemize}
there exists a birational morphism $\pi \colon W \to Z$ such that
\begin{itemize}
	\item $W$ is a smooth projective surface and $\Supp(\pi_*^{-1}(C+B)) \cup \Exc(\pi)$ is simple normal crossing, and 
	\item $(W, \Supp(\pi_*^{-1}(C+B)) \cup \Exc(\pi))$ lifts to characteristic zero over a smooth base.
\end{itemize}
\end{proposition}

\begin{proof}
We can apply the same argument as in the proof of Proposition~\ref{prop-bounded-case} 
after replacing $(X, B)$ (resp.\ Lemma~\ref{lem-bounded-family}) 
by $(Z, C+B)$ (resp.\ Lemma~\ref{l-bounded-family2}). 
\end{proof}

\begin{lemma} \label{l_blow_up_lift} 
Let $Z$ be a smooth projective surface  over an algebraically closed field $k$ of characteristic $p>0$ 
and let $D$ be a reduced simple normal crossing divisor on $Z$. 
Let $\pi \colon Z' \to Z$ be the blow up at a point $q$ contained in the singular locus of $D$ and 
let $E$  be the $\pi$-exceptional $(-1)$-curve. 
If $(Z,D)$ lifts to characteristic zero over a smooth base, 
then so does $(Z', \pi_*^{-1}D+E)$. 
\end{lemma}
\begin{proof}
We may write $D = \sum_{i=1}^n D_i$, where $D_1,\dots,D_n$ are the irreducible components. 
Since $(Z,D)$ lifts to characteristic zero, there exists a pair $(\mathcal{Z}, \mathcal{D}=\sum_{i=1}^n\mathcal D_i)$ over a smooth separated scheme $T$ over $\Spec \Z$ and a morphism $\alpha\colon \Spec\,k \to T$, 
such that the base changes of the schemes $\mathcal X, \mathcal D_1,\dots,\mathcal D_n$ over $T$ by $\alpha\colon \Spec\,k \to T$ are isomorphic to $X,D_1,\dots,D_n$ respectively and  satisfy the same properties as in Definition~\ref{d-liftable}.

By assumption, there exist unique $D_i, D_j$ such that $q \in D_i \cap D_j$. 
Let $\mathcal C$ be the irreducible component of $\mathcal D_i \cap \mathcal D_j$ 
such that $q \in \mathcal C$. 
Let $\overline{\pi} \colon \mathcal{Z}' \to \mathcal{Z}$ be the blow-up along $\mathcal C$. 
Then, $\mathcal{Z}'$ is a lift of $Z'$. Indeed, since $\mathcal{D}_i \cap \mathcal{D}_j$ is smooth over $T$, 
so is $\mathcal C$. 
It follows that 
$\mathcal{Z}'$ is smooth over $T$ and has connected fibres over $T$ 
(e.g.\ see the proof of 
\cite[Section 8, Theorem 1.19]{liu02}). 
By [ibid], the exceptional divisor $\mathcal E$ of $\overline \pi$ is a $\mathbb P^1$-bundle over 
$\mathcal C$. 
Thus, all the assumptions in Definition~\ref{d-liftable} hold true for the pairs 
$(Z', \pi_*^{-1}D+E)$ and 
 $(\mathcal{Z}', \overline{\pi}_*^{-1}\mathcal D+\mathcal E)$ over $T$.
\end{proof}


\begin{remark}\label{r_terminalisation}
By Lemma~\ref{l_blow_up_lift}, 
we may assume that 
the log resolution appearing in (2) of Theorem~\ref{thm-main} 
factors through the terminalisation $\eta \colon W\to X$  of $(X, B)$, i.e.\ if we write
\[
K_W+B_W=\eta^*(K_X+B),
\] 
then $(W,B_W)$ is terminal.
\end{remark}

\section{Non-$\epsilon$-klt case and proof of Theorem~\ref{thm-main}}\label{s_proof-main-thm}

In Section \ref{section-bounded-case}, we showed that $\epsilon$-klt log del Pezzo surfaces 
form a bounded family. 
The goal of this section is to study log del Pezzo surfaces $(X, B)$ over an algebraically closed field $k$, which are not $\epsilon$-klt for $0 < \epsilon \ll 1$. In particular, this yields a proof of 
our main Theorem (Theorem~\ref{thm-main}). 

Our method consists of constructing a log canonical pair $(Z,C_Z +B_Z)$ from $(X,B)$, such that $\lfloor C_Z +B_Z \rfloor = C_Z$ is prime and $-(K_Z + C_Z + B_Z)$ is nef.
If $K_Z+C_Z+B_Z$ is not numerically trivial, then, by using  global $F$-adjunction (see Subsection \ref{subsection_global-f-adjunction}), we show that $(Z,C_Z+B_Z)$ is globally $F$-split, provided that  the characteristic of $k$ is large enough. This in turn implies that  $(X,B)$ is globally $F$-regular.

Unfortunately, if $K_Z+C_Z+B_Z$ is numerically trivial, then $(Z,C_Z+B_Z)$ does not need to be globally $F$-split \cite[Theorem 1.1]{CTW15a}. Thus, we need to consider  two different cases, depending on whether  $(Z, C_Z+B_Z)$ is plt or not. 
If the pair is plt (resp.\ non-plt), then we show that  condition (2) (resp.\  condition (1)) 
of Theorem \ref{thm-main} holds.

\begin{proposition}\label{p_cases}
Let $I \subseteq (0, 1)\cap \mathbb Q$ be a finite set. 

Then there exists a rational number $\epsilon(I)>0$  which satisfies the following property:

Let $(X, B)$ be a two dimensional projective klt pair over an algebraically closed field  such that 
\begin{itemize}
\item $-(K_X+B)$ is ample,
\item  $(X,B)$ is not $\epsilon(I)$-klt, and 
\item  the coefficients of $B$ are contained in $I$.
\end{itemize}
Then there exist 
birational morphisms 
\[
\begin{CD}
Y @>g>> Z\\
@VVf V\\
X
\end{CD}
\]
of projective klt surfaces such that if we denote $C = \Exc(f)$, $C_Z = g_* C$, $B_Y = f^{-1}_* B$ and $B_Z = g_*B_Y$, then $C$ is a prime divisor, $C_Z \neq 0$, 
\begin{enumerate}
\item if $a=a(C,X,B)$ so that  
\[
K_Y + (1-a) C + B_Y = f^*(K_X + B),
\]
then $a\in (0, \epsilon(I)]$,  

\item there exists a $\mathbb Q$-divisor $B_Y^+ \geq B_Y$ such that 
\[
K_Y+C+B_Y^+=g^*(K_Z+C_Z+B_Z)
\]
and 
\[
\Exc(g) = \Supp (B_Y^+-B_Y),
\]
\item $-(K_Z+C_Z+B_Z)$ is nef, 
\item $(Z, C_Z+B_Z)$ is log canonical, and 
\item $(Z, bC_Z+B_Z)$ is klt and $-(K_Z+bC_Z+B_Z)$ is ample for some rational number $b \in (1-a, 1)$. 
\end{enumerate}

\end{proposition}

\begin{proof} 
Pick any rational number $\epsilon(I)\in (0,1)$ such that $\epsilon <1-c$ for any $c\in I$.  We will replace $\epsilon(I)$ by a smaller number, if necessary. 

Let $(X, B)$ be a log pair which satisfies the assumptions in the Proposition. 
Let $C$ be an exceptional  divisorial valuation such that $a := a(C, X,B)$ is minimal and let $f \colon Y \to X$ be a projective birational morphism such that $\Exc(f)= C$. By assumption, $a\in (0,\epsilon(I)]$. Thus, (1) holds. 

Since $-(K_X+B)$ is ample, we can find a rational number $b\in (1-a,1)$ such that  
$(Y, bC+B_Y)$ is klt, and $-(K_Y+bC+B_Y)$ is ample. We run a $-(K_Y+C+B_Y)$-MMP with scaling of $C$:
\begin{multline*}
(Y, C+B_Y)=:(Y_0, C_0+B_0) \xrightarrow{\hphantom{-}g_0\hphantom{1}} (Y_1, C_1+B_1) \xrightarrow{\hphantom{-}g_1\hphantom{1}} \cdots \\ 
\cdots \xrightarrow{g_{n-1}} (Y_n, C_n+B_n)=:(Z, C_Z+B_Z),
\end{multline*}
where $C_i := (g_{i-1})_*C_{i-1}$ and $B_i := (g_{i-1})_*B_{i-1}$. 

By definition of the MMP with scaling, we get a sequence of rational numbers 
\[b<b_0 \leq b_1 \leq \cdots \leq b_{n-1}\]
such that 
\[
b_i := \max\{t>b\mid -(K_{Y_i}+tC_i+B_i) \text{ is nef}\}
\]
and $g_i\colon Y_i\to Y_{i+1}$ is a birational morphism, such that $\rho(Y_i/Y_{i+1})=1$ and 
\[
K_{Y_i}+b_iC_i+B_i=g_i^*(K_{Y_{i+1}}+b_iC_{i+1}+B_{i+1}).
\]

Since $-(K_{Y_i}+bC_i+B_i)$ is ample for any $i=0,\dots,n$, each step of 
the MMP is $(K_{Y_i}+bC_i+B_i)$-negative, and in particular $(Y_i, bC_i+B_i)$ is klt. 
Thus, (5) holds. 

Since the coefficients of $B_Z$ belong to the finite set $I$ and $(Z, bC_Z+B_Z)$ is klt, 
by ACC for the log canonical threshold in dimension two \cite[Theorem 1.1]{alexeev93}, after possibly replacing $\epsilon(I)$ by a smaller value depending only on $I$, we may assume that 
the pair $(Z, C_Z + B_Z)$ is log canonical, hence (4) holds. 

Since $C_i$ is $g_i$-ample for all $i=0,\dots,n-1$, it follows that $C_Z \neq 0$. 
We may write 
\[
g^*(K_Z+C_Z+B_Z)=K_Y+C+B_Y+R
\]
for some $g$-exceptional $\mathbb Q$-divisor $R$ on $Y$. 
Since the MMP is $-(K_Y+C+B_Y)$-negative, it follows that $R\ge 0$ and the support of $R$ coincides with the exceptional locus of $g$. In particular, if $B^+_Y=B_Y+R$, then    (2) holds. 

Thus, it is enough to show (3). 
We assume by contradiction that $-(K_Z+C_Z+B_Z)$ is not nef. 
Thus,  there exists a $-(K_Z+C_Z+B_Z)$-negative Mori fibre space  $\pi\colon Z \to T$ and,
in particular, 
\[
(K_Z+C_Z+B_Z) \cdot F>0
\]
for any curve $F$ contained in a fibre of $\pi$. 
On the other hand, since $-(K_Z+bC_Z+B_Z)$ is ample, it follows that 
\[
(K_Z+bC_Z+B_Z) \cdot F<0.
\]
Thus, $C_Z$ is $\pi$-horizontal and we can find a rational number $b'\in (b,1)$ such that
\[
K_Z+b'C_Z+B_Z \equiv_{\pi} 0.
\]
Since $1-\epsilon(I)<b<b'<1$, 
after possibly replacing $\epsilon(I)$ by a smaller number depending only on $I$, we get a contradiction by Lemma \ref{l_acc.num}. Thus, (3) holds. 
\end{proof}

Before we proceed with the proof of Theorem \ref{thm-main}, we recall the following criterion for  global $F$-regularity:

\begin{lemma}\label{l_gfr}
Let $(Z,C+B)$ be a log pair, where $\lfloor C + B \rfloor = C$ is a reduced divisor. Let $H$ be an effective $\mathbb Q$-divisor on $Z$ such that $Z\, \backslash\, \Supp (H + C)$ is affine. Assume that
\begin{enumerate}
\item $(Z,B)$ is strongly $F$-regular, and
\item there exists $\delta > 0$ such that $(Z,C + B + \delta H)$ is globally $F$-split.
\end{enumerate}
Then $(Z, \lambda C+B)$ is globally $F$-regular for every $\lambda \in [0,1)$.
\end{lemma}

\begin{proof}
This follows from \cite[Theorem 3.9]{schwedesmith10}.
\end{proof}

We now prove Theorem \ref{thm-main}.

\begin{proof}[Proof of Theorem \ref{thm-main}]
Let $\epsilon(I)>0$ be the positive rational number as in Proposition \ref{p_cases} and 
let $p(I):=p(I,\epsilon(I))$ be the positive integer as in Proposition \ref{prop-bounded-case}. 
We will  replace $\epsilon(I)$ (resp.\ $p(I)$)  by a smaller (resp.\ larger) value if necessary. 

Let $(X, B)$ be a log del Pezzo pair over an algebraically closed field of characteristic $p>p(I)$. 
If $(X, B)$ is $\epsilon(I)$-klt, then Proposition \ref{prop-bounded-case} implies (2) of Theorem \ref{thm-main}. 
Thus, we may assume that $(X, B)$ is not $\epsilon(I)$-klt. 
By Proposition \ref{p_cases}, there exist birational morphisms 
\[
\begin{CD}
Y @>g>> Z\\
@VVf V\\
X
\end{CD}
\]
which satisfy the properties of Proposition~\ref{p_cases}. 

\medskip
\textbf{Claim:} 
After possibly replacing $p(I)$ by a larger value depending only on $I$, the following  hold: 
\begin{enumerate}
\item[(a)]{$(p^e-1)B_Z$ is a $\mathbb Z$-divisor for some $e \in \mathbb Z_{>0}$,}
\item[(b)]{$(Z, C_Z+B_Z)$ is sharply $F$-pure and 
$(Z, \lambda C_Z+B_Z)$ is strongly $F$-regular for any $\lambda \in [0, 1)$, and}
\item[(c)]{if $(Z, \lambda C_Z+B_Z)$ is globally $F$-regular for any $\lambda \in [0, 1)$, 
then $(X, B)$ is globally $F$-regular.}
\end{enumerate} 

We now prove the Claim. 
(a) is clear. 
 (4) and (5) of Proposition \ref{p_cases} imply that $(Z, \lambda C_Z + B_Z)$ is klt for any $\lambda \in [0, 1)$. Thus, after possibly replacing  $p(I)$ by a larger number depending only on $I$, by \cite[Theorem 1.1]{cgs14}, we may assume that $(Z,\lambda C_Z + B_Z)$ is strongly $F$-regular. 
By \cite[Proposition 3.3]{CTW15a}, $(Z, C_Z+B_Z)$ is $F$-pure, hence it is sharply $F$-pure by (a). 
Thus  (b) holds. 

We now show (c). 
Assume that $(Z, \lambda C_Z + B_Z)$ is globally $F$-regular for any $\lambda \in (0, 1)$. 
By (2) of Proposition \ref{p_cases}, if $\lambda \in (0, 1)$ is a rational number which is sufficiently close to $1$, then 
\[
K_Y+\lambda C+B'_Y =g^*(K_Z+\lambda C_Z+B_Z),
\]
for some $\mathbb Q$-divisor $B'_Y\ge B_Y$. 
By \cite[Proposition 2.11]{hx13}, it follows that  $(Y, \lambda C+B_Y)$ is globally $F$-regular for some $\lambda \in (0, 1)$. 
By \cite[Lemma 2.2]{CTW15a}, it follows that $(X, B)$ is  globally $F$-regular as well. Thus,  
 (c) holds and this completes the proof of the Claim.

\medskip

By (3) of Proposition \ref{p_cases}, it follows that $-(K_Z+C_Z+B_Z)$ is  nef. By \cite[Theorem 1.2]{tanaka12}, it is semi-ample. 
In order to prove the Theorem, we consider the following four cases separately: 
\begin{itemize}
\item{$\kappa(Z, -(K_Z+C_Z+B_Z))=2$,}
\item{$\kappa(Z, -(K_Z+C_Z+B_Z))=1$,}
\item{$\kappa(Z, -(K_Z+C_Z+B_Z))=0$ and $(Z, C_Z+B_Z)$ is not plt,}
\item{$\kappa(Z, -(K_Z+C_Z+B_Z))=0$ and $(Z, C_Z+B_Z)$ is plt.}
\end{itemize}

If  one of the first three cases  (resp.\ if the last case) holds, 
then we will show that  (1) (resp.\ (2)) of Theorem \ref{thm-main} holds. 
If $C_Z$ is normal, then we define $(C_Z,B_{C_Z})$ by adjunction
\[
(K_Z + C_Z + B_Z)|_{C_Z} = K_{C_Z} + B_{C_Z}.
\]

\medskip 

\textbf{Case 1:} $\kappa(Z, -(K_Z+C_Z+B_Z))=2$.

Since $-(K_Z+C_Z+B_Z)$ is semi-ample,  
we can contract the curves $\Gamma$ with $(K_Z+C_Z+B_Z)\cdot \Gamma=0$. 
In this contraction, $C_Z$ is not contracted, because if $(K_Z+C_Z+B_Z) \cdot C_Z=0$, 
then $C_Z^2<0$, which contradicts 
$$(1-b)C_Z^2=(K_Z+C_Z+B_Z) \cdot C_Z-(K_Z+bC_Z+B_Z) \cdot C_Z >0.$$
Thus, by \cite[Proposition 2.11]{hx13}, we may assume that $-(K_Z+C_Z+B_Z)$ is ample. 
By \cite[Theorem 3.19]{tanaka12} and 
$$(K_Z+C_Z) \cdot C_Z\leq (K_Z+C_Z+B_Z) \cdot C_Z<0,$$
it follows that $C_Z \simeq \mathbb P^1$. 

By Lemma \ref{l_perturbation}, there exists an effective $\mathbb{Q}$-divisor $H$
such that $(Z, C_Z+B_Z+H)$ is log canonical 
and the Cartier index of $K_Z + C_Z + B_Z + H$ is not divisible by $p$. 
After replacing $H$ by $H+H'$ for some ample effective $\mathbb Q$-divisor $H'$, 
we may assume that $Z \setminus\Supp H$ is affine. After
replacing $H$ by a smaller multiple, we may assume that $-(K_Z + C_Z + B_Z + H)$ is ample.

We now show that $(C_Z, B_{C_Z}+H|_{C_Z})$ is globally $F$-split. 
If $(C_Z,B_{C_Z}+H|_{C_Z})$ is not klt, then Lemma \ref{l_gladj} implies that  $(C_Z, B_{C_Z} + H|_{C_Z})$ is globally $F$-split. 
If $(C_Z,B_{C_Z}+H|_{C_Z})$ is klt, then by \cite[Corollary 4.1]{cgs14}, after possibly replacing $p(I)$ by a larger value depending only on $I$,  we may assume that $(C_Z,B_{C_Z})$ is globally $F$-regular. After possibly replacing $H$ by a smaller multiple again, we may assume that $(C_Z, B_{C_Z} + H|_{C_Z})$ is globally $F$-split.

By Lemma \ref{l_gladj}, it follows that $(Z, C_Z+B_Z+H)$ is globally $F$-split. 
 Thus, (b) of the Claim and Lemma \ref{l_gfr} imply that $(Z, \lambda C_Z+B_Z)$ is globally $F$-regular for any $\lambda\in [0,1)$. By (c) of the Claim,  it follows that $(X, B)$ is globally $F$-regular.

\medskip 

\textbf{Case 2:} $\kappa(Z, -(K_Z+C_Z+B_Z))=1$.

Since $-(K_Z+C_Z+B_Z)$ is semi-ample, there exist a morphism  $\pi \colon Z \to T$ onto a smooth projective curve $T$ and 
 an ample $\mathbb Q$-divisor $A$ on $T$ such that $\pi_*\mathcal O_Z=\mathcal O_T$ and
$-(K_Z+C_Z+B_Z) \equiv \pi^*A$.  
By (5) of Proposition~\ref{p_cases} and  by \cite[Proposition 2.11]{hx13}, 
after running a $(K_Z + bC_Z + B_Z)$-MMP over $T$, we may assume that 
$\pi \colon Z \to T$ is a $(K_Z + bC_Z + B_Z)$-negative Mori fibre space. In particular, $C_Z$ is $\pi$-horizontal. 
By \cite[Theorem 3.19]{tanaka12} and 
$$(K_Z+C_Z) \cdot C_Z\leq (K_Z+C_Z+B_Z) \cdot C_Z<0,$$
it follows that $C_Z \simeq \mathbb P^1$. 

By (a) of the Claim, we can apply Lemma \ref{l_mfs-perturbation}, and 
there exists an effective $\mathbb Q$-divisor $E$, whose support is contained in some fibres of $\pi$,  and such that $(Z, C_Z+B_Z+E)$ is log canonical and the Cartier index of $K_Z + C_Z + B_Z + E$ is not divisible by $p$. 
After possibly replacing $E$ by a smaller multiple, we may assume that 
$-(K_Z+C_Z+B_Z+E) \equiv \pi^*A'$ for some ample $\mathbb Q$-divisor $A'$ on $T$.
In particular, $-(K_{C_Z}+B_{C_Z}+E|_{C_Z})$ is ample.

By the same argument as in Case 1, $(C_Z, B_{C_Z}+E|_{C_Z})$ is globally $F$-split. 
Thus, by Lemma~\ref{lem-MFS}, $(Z, C_Z+B_Z+E)$ is globally $F$-split. 
Again by the same argument as in Case 1, it follows that $(X, B)$ is globally $F$-regular.

\medskip 

\textbf{Case 3:} $\kappa(Z, -(K_Z+C_Z+B_Z))=0$ and $(Z, C_Z+B_Z)$ is not plt.

By Proposition \ref{p_bounded-index}, after possibly replacing $p(I)$ by a larger value depending only on $I$, we may assume that 
$(p^e-1)(K_Z + C_Z + B_Z)$ is Cartier for some positive integer $e$. 
Thus, Lemma \ref{l_pic1-non-plt} implies that $(Z, C_Z+B_Z)$ is globally $F$-split. 
By (5) of Proposition~\ref{p_cases}, it follows that   $-(K_Z + bC_Z + B_Z)$ is ample, and in particular also  $C_Z$ is ample. 
Thus, (b) of the   Claim  and Lemma \ref{l_gfr} imply that $(Z, \lambda C_Z+B_Z)$ is globally $F$-regular for any $\lambda \in [0,1)$. 
By (c) of the Claim, it follows that  $(X, B)$ is globally $F$-regular.

\medskip 

\textbf{Case 4:} $\kappa(Z, -(K_Z+C_Z+B_Z))=0$ and $(Z, C_Z+B_Z)$ is  plt.

By Proposition~\ref{p_log_resol_lift}, 
after possibly replacing $p(I)$ by a larger number depending only on $I$, 
there exists a log resolution $\pi \colon W \to Z$ of $(Z, C+B)$ 
such that 
\[
(W, D_W:=\Supp(\pi_*^{-1}(C+B)) \cup \Exc(\pi))
\] 
lifts to characteristic zero over a smooth base. 
We may write 
\[
K_W+D'_W=\pi^*(K_Z+C_Z+B_Z).
\]
for some $\mathbb Q$-divisor $D'_W$ on $W$. 
Note that $\Supp D'_W \subseteq \Supp D_W$. 
By (2) of Proposition~\ref{p_cases}, 
the birational morphism $g\colon Y \to Z$ only extracts prime divisors $E$ such that 
$a(E, W, D'_W)=a(E, Z, C_Z+B_Z)<1$. 
Thus, there exists a sequence of blow-ups 
\[
\varphi\colon V:=W_m\xrightarrow{\varphi_{m-1}} \cdots \xrightarrow{\varphi_1} W_1\xrightarrow{\varphi_0} W_0:=W
\] 
which satisfies the following properties: 
\begin{itemize}
\item{We define a divisor $D_i$ on $W_i$ inductively by  
\begin{eqnarray*}
D_0&:=&D_W\\
K_{W_{i+1}}+D_{i+1}&=&\varphi_i^*(K_{W_{i}}+D_{i}).
\end{eqnarray*}}
\item{Each $\varphi_i$ is the blow-up of a point contained in $\Sing(\Supp D_{i})$.}
\item{The composite arrow $\psi \colon V \xrightarrow{\varphi} W \xrightarrow{\pi} Z$ factors through $Y$. }
\end{itemize}
Note that each $D_i$ is a reduced simple normal crossing divisor. 
Therefore, by Lemma~\ref{l_blow_up_lift}, $(V, \psi_*^{-1}(C_Z+B_Z) \cup \Exc(\psi))$ 
lifts to characteristic zero over a smooth base. 
Thus, (2) of the Theorem holds. 
\end{proof}

\section{Kawamata--Viehweg vanishing for log del Pezzo surfaces}\label{s_kv}

The goal of this section is to prove Theorem \ref{t_kv}. 
We begin with the following:

\begin{lemma}\label{l_W2_vanishing}
Assume that $\mathrm{char}\,k>2$. 
Let $(X, \Delta)$ be a two dimensional projective klt pair. 
Suppose that there exists a log resolution $\mu\colon V \to X$ 
of $(X, \Delta)$ such that 
\[
(V, \Exc(\mu) \cup \mu_*^{-1}(\Supp \Delta))
\]
lifts to $W_2(k)$. 
Let $D$ be a 
$\Z$-divisor on $X$ such that $D-(K_X+\Delta)$ is ample.

Then,
\[
H^i(X, \mathcal O_X(D))=0 \qquad \text{for any }i>0.
\]
\end{lemma}

%

\begin{proof}
Let 
\[
M=K_{V}+\ulcorner\mu^*(D-(K_X+\Delta)) \urcorner.
\]
We may find a $\mu$-exceptional $\Q$-divisor $E\ge 0$ such that 
\begin{itemize}
\item{$M=K_{V}+\ulcorner \mu^*(D-(K_X+\Delta))-E \urcorner,$ and}
\item{$\mu^*(D-(K_X+\Delta))-E$ is ample.}
\end{itemize}
Note that $\{\mu^*(D-(K_X+\Delta))-E\}$ is simple normal crossing and lifts to $W_2(k)$. 
Since $\mathrm{char}\,k>2$, \cite[Corollary 3.8]{hara98a} implies that 
\[
H^i(V, \mathcal O_V(M))=0
\]
for every $i>0$. 
Consider the  Leray spectral sequence: 
$$E^{i,j}_2:=H^i(X, R^j\mu_*\mathcal O_{V}(M))
\Rightarrow H^{i+j}(V, \mathcal O_{V}(M))=:E^{i+j}.$$
We have
\begin{eqnarray*}
&&\mu_*\mathcal O_{V}(M)\\
&=&\mu_*\mathcal O_{V}(\ulcorner \mu^*(D)+K_V-\mu^*(K_X+\Delta) \urcorner)\\
&=&\mu_*\mathcal O_{V}(\llcorner\mu^*(D)\lrcorner+
(\mathrm{effective\,\,exceptional\,\,\mathbb Z\text{--}divisor}))\\
&\simeq&\mathcal O_{X}(D),
\end{eqnarray*}
where the  second equality holds 
because $(X, \Delta)$ is klt and $D$ is a $\Z$-divisor. 
Since $E_2^{i, j}=0$ for $j>0$ by the relative Kawamata--Viehweg vanishing theorem  for a proper birational morphism between surfaces, we obtain 
$$H^i(X, \mathcal O_X(D))=E_2^{i, 0} \simeq E^i =H^i(V, \mathcal O_V(M))=0$$
for $i>0$. Thus, the claim follows. 
\end{proof}

\begin{proof}[Proof of Theorem \ref{t_kv}]
After possibly perturbing $B$ and $\Delta$, we may assume that $-(K_X+B)$ and $D-(K_X+\Delta)$ are ample. 
Given $I:=\emptyset$, let $p_0:=p(I)$ be the positive integer whose existence is guaranteed by Theorem~\ref{thm-main}. We divide the proof into five steps:

\medskip

\setcounter{step}{0}
\begin{step} We first prove the Theorem assuming that  $\Delta=B=0$.
\end{step}

Theorem \ref{thm-main} implies that $X$ is globally $F$-regular or $X$ admits a 
 log resolution $\mu\colon V \to X$ such that $(V, \Exc(f))$ lifts to $W_2(k)$. 
In the first case, by Serre vanishing, if $e$ is a sufficiently large positive integer and $i>0$, we have 
\[
\begin{aligned}
H^i(X, \mathcal O_X(D)) \hookrightarrow &~ H^i(X, F_*^e\mathcal O_X(p^eD-(p^e-1)K_X)) \\
\simeq &~ H^i(X, \mathcal O_X(D+(p^e-1)(D-K_X))) =0.
\end{aligned}
\]
On the other hand, if $X$ admits a log resolution $\mu \colon V \to X$ such that $(V, \Exc(f))$ lifts to $W_2(k)$, 
then  Lemma \ref{l_W2_vanishing} implies the claim. 

\medskip

\begin{step}
We now prove  the Theorem under the assumption that $D$ is nef. 
\end{step}

Let $f\colon X\to Y$ be a proper birational contraction obtained by running a 
$(-K_X)$-MMP. Since $(X,B)$ is log del Pezzo, it follows that $Y$ is klt and  $-K_Y$ is  big and semi-ample. 
Moreover, there exists an $f$-exceptional $\Q$-divisor $\Delta_1\ge 0$ such that 
\[
K_X+\Delta_1 =f^*K_Y.
\]
Let $g\colon Y \to Z$ be the morphism induced by $-K_Y$ and let 
\[
h\colon X \xrightarrow{f} Y \xrightarrow{g} Z
\]
be the composite morphism. 
Then $-K_Z$ is ample and 
\[
K_X+\Delta_1=h^*K_Z.
\]
By Theorem \ref{thm-main},  $Z$ is globally $F$-regular or it admits a log resolution 
$\mu\colon V \to Z$  such that 
$(V, \Exc(\mu))$ lifts to characteristic zero over a smooth base. 

If $Z$ is globally $F$-regular, then  \cite[Proposition 2.11]{hx13} implies that  $(X, \Delta_1)$ is also globally $F$-regular. 
Thus, we may find a $\Q$-divisor $\Delta_2 \geq \Delta_1$ such that 
$(X, \Delta_2)$ is globally $F$-regular, $-(K_X+\Delta_2)$ is ample, and 
$(p^{e_1}-1)(K_X+\Delta_2)$ is Cartier for some $e_1 \in \mathbb Z_{>0}$. 
Since $D$ is nef, by Fujita vanishing,  if $e$ is a sufficiently divisible positive integer, for any $i>0$ we have 
\[
H^i(X, \mathcal O_X(D)) 
\hookrightarrow H^i(X, F_*^e\mathcal O_X(p^eD-(p^e-1)(K_X+\Delta_2)))=0,
\]
and the Theorem follows. 

Thus, we may assume the existence of the  log resolution $\mu$, as above. 
Since $\Delta_1 \geq 0$, the birational morphism $h\colon X \to Z$ only extracts prime divisors $E$ 
with $a(E, Z, 0) \in (0, 1]$. 
By Remark~\ref{r_terminalisation}, we may assume that $\mu\colon V \to Z$ factors through $X$, 
after possibly replacing $V$ by some blow-up of $V$: 
\[
\mu\colon V \xrightarrow{\mu_X} X \xrightarrow{h} Z.
\]
There exists an $h$-exceptional $\mathbb Q$-divisor $\Delta_2 \geq \Delta_1$ such that 
$(X, \Delta_2)$ is klt and  $-(K_X+\Delta_2)$ is ample.
Since $\mu_X^{-1}(\Delta_2) \subseteq \Exc(\mu)$, 
 Lemma~\ref{l_W2_vanishing} implies the Theorem.

\medskip

\begin{step}
We now show that we may assume that there exists a $D$-negative Mori fibre space  $g\colon X \to Z$ 
onto a smooth projective curve $Z$. 
\end{step}

Let $f\colon X \to Y$ be the birational contraction of a curve $E$ such that $D \cdot E \leq 0$. 
We want to show 
\[
H^i(X, \mathcal O_X(D))=H^i(Y, \mathcal O_Y(f_*D)).
\]
Let $D_Y=f_*D$. Then $D=f^*D_Y+cE$ for some $c \in \mathbb Q_{\geq 0}$, 
which implies 
\[
f_*\mathcal O_X(D) =\mathcal O_Y(D_Y).
\]
Since $D-(K_X+\Delta)$ is ample,  the relative Kawamata--Viehweg vanishing in the birational case implies 
\[
H^i(X, \mathcal O_X(D)) \simeq H^i(Y, f_*\mathcal O_X(D)) =H^i(Y, \mathcal O_Y(D_Y)).
\]

Since $(X,B)$ is log del Pezzo, we may run a $D$-MMP and may assume that one of the following  holds: 
\begin{itemize}
\item{$D$ is nef,}
\item{$\rho(X)=2$, and there exists a $D$-negative Mori fibre space  $X \to Z$ 
onto a smooth projective curve $Z$, or}
\item{$\rho(X)=1$, and $-D$ is ample.}
\end{itemize}
If $D$ is nef, then we may apply Step 2. 
In the last case, we may assume that $\Delta=B=0$, and we may apply Step 1. 
Thus, we may assume that there exists a $D$-negative Mori fibre space  $g\colon X \to Z$ onto a smooth projective curve $Z$. 

\medskip

\begin{step}
We now show that we may assume that the following  hold: 
\begin{enumerate}
\item[(a)]{There exists a curve $E$ on $X$ such that $E^2<0$,}
\item[(b)]{$B=bE$ for some $b\in [0,1)$, and }
\item[(c)]{$\Delta=0$.}
\end{enumerate}
\end{step}


If every curve $C$ on $X$ is nef, 
then we may assume that $\Delta=B=0$, and apply Step 1. 
Thus, we may assume that there exists a curve $E$ on $X$ such that $E^2<0$ and (a) holds. 

By Step 3, we may assume that there exists a $D$-negative Mori fibre space $g\colon X\to Z$ onto a smooth projective curve $Z$.  In particular,  $\rho(X)=2$,  and  any  curve $C$ on $X$, which is different from $E$, is nef. Thus, we may assume that (b) holds.
Similarly, we may assume that $\Delta=\delta E$ for some $\delta\in [0,1)$. 
Let $F$ be a general fibre of $g$. Then $D \cdot F<0$. 
Consider the short exact sequence
\[
0 \to \mathcal O_X(D) \to \mathcal O_X(D+F) \to \mathcal O_F(D|_F) \to 0.
\]
Since $H^0(F, \mathcal O_F(D|_F))=H^1(F, \mathcal O_F(D|_F))=0$, 
it is enough to show that $H^i(X, \mathcal O_X(D+F))=0$ for any $i>0$. 
Repeating the same method finitely many times, it is enough to show that  $H^i(X,\mathcal O_X(D+nF))=0$ for some positive integer $n$ and for any $i>0$. 
Since $D-(K_X+\delta E)$ is ample, 
the divisor 
\[
D+nF-K_X=(D-(K_X+\delta E))+(nF+\delta E)
\]
is also ample for $n \gg 0$. 
Therefore, after possibly replacing $D$ by $D+nF$, we may assume that $\Delta=0$ and (c) holds.  

\medskip

\begin{step}
We now prove the Theorem in the general case. \end{step}

We may assume that (a), (b) and (c) of Step 4 hold. If $K_X \cdot E<0$, then $-K_X$ is ample by Kleiman's criterion and we may assume that $B=0$. Thus, the Theorem follows by Step 1. 

Therefore,  we may assume that $K_X \cdot E \geq 0$. 
Let $f\colon X \to Y$ be the birational morphism which contracts $E$. 
We may write 
\[
K_X+b' E=f^*K_Y
\]
for some $b' \in [0,b)$. 
Thus, by Theorem \ref{thm-main} and Remark~\ref{r_terminalisation}, 
$Y$ is globally $F$-regular or it admits a log resolution $V \to Y$ which factors through $X$ 
such that $(V, \Exc(\mu))$ lifts to $W_2(k)$: 
\[
\mu\colon V \xrightarrow{\mu_X} X \to Y.
\]
In the latter case, since $\Exc(\mu_X) \subseteq \Exc(\mu)$, we may apply Lemma \ref{l_W2_vanishing}. 
Thus, we may assume that $Y$ is globally $F$-regular.   
By \cite[Prop 2.11]{hx13}, 
$(X, b' E)$ is globally $F$-regular, and so is $X$. 
Step 4(c) implies that  $D-K_X$ is ample, 
Therefore, 
by Serre vanishing, if $e$ is  a sufficiently large positive integer, we have
\[
H^i(X, D) \hookrightarrow H^i(X, p^eD-(p^e-1)K_X)
=H^i(X, K_X+p^e(D-K_X))=0.
\]
Thus, the Theorem follows. 
\end{proof}

\section{Example in characteristic two}\label{s_char2_example}

The goal of this section is to show Theorem~\ref{t_char2_example}. We begin with the following two preliminary results: 

\begin{lemma}\label{l_weil_conj}
Let $T \to \Spec\, \mathbb Z$ be a smooth morphism from an integral scheme $T$
and let $\mathcal W \to T$ be a flat projective morphism. 
Fix a morphism $\alpha\colon \Spec\,\overline{\mathbb F}_{p} \to T$ and an embedding $K(T) \hookrightarrow \mathbb C$. 
Let $W_{\overline{\mathbb F}_p}:=\mathcal W \times_{T} \overline{\mathbb F}_{p}$, and 
$W_{\mathbb C}:=\mathcal W \times_{T} \mathbb C$. 
If $W_{\overline{\mathbb F}_p}$ is a smooth projective rational surface, 
then the following  hold: 
\begin{enumerate}
\item{$W_{\mathbb C}$ is a smooth projective rational surface, and}
\item{$\rho(W_{\mathbb C})=\rho(W_{\overline{\mathbb F}_p})$.}
\end{enumerate}
\end{lemma}

\begin{proof}
Since being smooth and geometrically integral are open properties, it follows that 
$W_{\mathbb C}$ is a smooth projective surface.

By  upper semi-continuity, it follows that 
$H^i(W_{\mathbb C}, \mathcal O_{W_{\mathbb C}})=0$ for any $i>0$ and 
$H^2(W_{\mathbb C}, \mathcal O_{W_{\mathbb C}}(-K_{W_{\mathbb C}}))=0$. 
Thus,  by the rationality criterion, 
it follows that $W_{\mathbb C}$ is a rational surface and (1) holds. 

Note  that, for any sufficiently large positive integer $e$, the morphism $\alpha$ factors as follows 
\[
\alpha\colon \Spec\,\overline{\mathbb F}_{p} \to \Spec\,\mathbb F_{p^e} \to T.
\]
Let $W_{\mathbb F_{p^e}}:=\mathcal W \times_T \mathbb F_{p^e}$. 
After possibly replacing $e$ by a larger number, we may assume that 
$W_{\mathbb F_{p^e}}$ is 
obtained as a sequence of blow-ups of 
$\mathbb P^2_{\mathbb F_{p^e}}$ or a $\mathbb P^1$-bundle over $\mathbb P^1_{\mathbb F_{p^e}}$,  
whose  centres are $\mathbb F_{p^e}$-rational points. 
By counting the rational points, 
it is easy to check that the zeta function of $W_{\mathbb F_{p^e}}$ can be written as
\[
Z_{W_{\mathbb F_{p^e}}}(t)=\frac{1}{(1-t)(1-p^et)^{\rho(W_{\overline{\mathbb F}_{p}})}(1-p^{2e}t)}.
\]
On the other hand, we obtain 
\[
\rho(W_{\overline{\mathbb F}_{p^e}})=\deg (1-p^et)^{\rho(W_{\overline{\mathbb F}_{p}})}=\dim_{\mathbb C} H^2(W_{\mathbb C}, \mathbb C)=\rho(W_{\mathbb C}),
\]
where 
the second equality follows from 
a consequence of 
the Weil conjecture  \cite[Chapter IV, Remark (b) after Theorem 1.2]{FK88}. Thus, (2) holds.  
\end{proof}

\begin{lemma}\label{l_non_gfs_const}
Let $k$ be an algebraically closed field of characteristic two. 
Then there exist  six distinct points $q_0, \dots, q_5 \in \mathbb P_k^2$ 
which satisfy the following properties: 
\begin{enumerate}
\item{if $g\colon Z \to \mathbb P_k^2$ is the blow-up at $q_0,\dots,q_5$,  
then $Z$ is a smooth del Pezzo surface,}
\item{$Z$ is not globally $F$-split, and }
\item{if $C$ is the smooth conic passing through $q_1, \ldots, q_5$,  
then, for every point $q\in C$, the line $L$ passing through $q_0$ and $q$ is tangent to $C$.}
\end{enumerate}
\end{lemma}

\begin{proof}
The smooth cubic surface $Z$ of Fermat type is not globally $F$-split 
 \cite[Example 5.5]{hara98a}. 
We can find six points $q_0, \dots, q_5 \in \mathbb P_k^2$ such that 
the blow-up along these points is isomorphic to $Z$. Let $g\colon Z\to \mathbb P_k^2$ be the induced morphism. 
Then (1) and (2) hold. 

By contradiction, we assume that 
there exists a line $L$ on $\mathbb P^2_k$, passing through $q_0$, such that $C+L$ is simple normal crossing. 
Let $C_Z$ and $L_Z$ be  the proper transforms on $Z$ of $C$ and $L$, respectively. 
By (1), there exists a point $q \in C \cap L$ with $q \neq q_i$ for any $i=1,\dots,5$, 
and, in particular,  $C_Z \cap L_Z \neq \emptyset$. 
There exists a reduced $g$-exceptional divisor $E_Z$ on $Z$ such that 
\[
K_Z+C_Z+L_Z+E_Z=g^*(K_{\mathbb P^2_k}+C+L) \sim 0,
\]
and  $(Z, C_Z+L_Z+E_Z)$ is log canonical but not plt. 
By Lemma~\ref{l_pic1-non-plt}, $(Z, C_Z+L_Z +E_Z)$ is globally $F$-split. Thus,  $Z$ is globally $F$-split, contradicting (2). 
 Thus, (3) holds. 
\end{proof}

We now prove the main result of this section.

\begin{proof}[Proof of Theorem~\ref{t_char2_example}]
Let $q_0, \dots, q_5 \in \mathbb P^2_{\overline{\mathbb F}_2}$ be as in Lemma~\ref{l_non_gfs_const}. 
For any $i=1,\dots,5$, let $L_i$ be the line in $\mathbb P^2_{\overline{\mathbb F}_2}$ 
passing through $q_0$ and $q_i$. 
Note that each $L_i$ is tangent to $C$ by (3) of Lemma~\ref{l_non_gfs_const}. 
Let 
\[
\pi\colon Y \to \mathbb P^2_{\overline{\mathbb F}_2}
\]
be the birational morphism constructed as follows. 
First, we consider the  blow-up $Z \to \mathbb P^2_{\overline{\mathbb F}_2}$ at the points $q_0,\dots,q_5$.  
Then, we consider the blow-up $Y\to Z$ at the points $q'_1,\dots,q'_5$, where $q'_i$ is the intersection point of the proper transforms on $Z$ of $C$ and $L_i$. 
In particular,  $\rho(Y)=12$. 
By the same argument as in \cite[Section 9, An interesting example in non-zero characteristic]{km99}, 
we can find curves $E_2, E_3, \dots, E_{12}$ on $Y$ with 
\[
E_2^2=-6, \qquad E^2_m=-2, \qquad E_i \cdot E_j=0
\]
for any $m=3,\dots,12$ and distinct $i,j= 2,\dots, 12$. 
Let $f\colon Y \to X$ be the birational morphism which contracts all these curves. 
Then $X$ is a projective klt surface with $\rho(X)=1$. 
By [ibid], $-K_X$ is ample. 
Thus, (1) holds. 

By (2) of Lemma~\ref{l_non_gfs_const}, 
$Z$ is not globally $F$-split. By  \cite[Lemma 2.2]{CTW15a}, it follows that $Y$ is not globally $F$-split and by 
 \cite[Proposition 2.11]{hx13}, it follows that  $X$ is not globally $F$-split. 
Thus, (2) holds. 

We now show (3). 
By contradiction, we assume that 
there exists a log resolution $h\colon W \to X$ such that 
$(W, \Exc(h)=\sum_{i=2}^{n} C_i)$ lifts to characteristic zero over a smooth base, i.e.\ 
there exists a smooth morphism $T \to \Spec\,\mathbb Z$, 
a closed immersion of schemes $\mathcal C_i \subseteq \mathcal W$ projective and flat over $T$ for any $i=2,\dots,n$, and 
a morphism $\alpha\colon \Spec\, \overline{\mathbb F}_2 \to T$ 
such that the base changes of $\mathcal C_i \subseteq \mathcal W$ by $\alpha$ is $C_i \subseteq W$ 
and such that $\mathcal W$ and all the strata of $\{\mathcal C_i\}$ are smooth over $T$. 
We may assume that $T$ is an integral scheme. 
Note that $K(T)$ is of characteristic zero. 
Fix an embedding $K(T) \hookrightarrow \mathbb C$, and let 
$W_{\mathbb C}:=\mathcal W \times_T \mathbb C$. 
By Lemma~\ref{l_weil_conj}, we have 
\[
n=\rho(W)=\rho(W_{\mathbb C}).
\]

The birational morphism $f\colon Y \to X$ constructed above is the minimal resolution of $X$. Thus, $h$ factors as
\[
h\colon W \xrightarrow{g} Y \xrightarrow{f} X
\]
and may assume that $g_*C_i=E_i$ for any $i=2,\dots,12$. 
We can find a sequence of blow-ups 
\[
g\colon W=:W_n \xrightarrow{g_n} W_{n-1} \xrightarrow{g_{n-1}} \dots \xrightarrow{g_{13}} W_{12}:=Y
\]
with $\rho(W_r)=r$ for any $r=12,\dots,n$. 
We may assume that the proper transform of $\Exc(g_r)$ on $W$ is $C_r$.

By  invariance of the intersection numbers, 
the two intersection matrices $(C_i \cdot C_j)$ and $(C_{i, \mathbb C} \cdot C_{j, \mathbb C})$ coincide. 
Thus, we can construct the corresponding sequence over $\mathbb C$, i.e.\ 
a sequence of blow-ups: 
\[
g_{\mathbb C}\colon W_{\mathbb C}=:W_{n, \mathbb C} \xrightarrow{g_{n, \mathbb C}} W_{n-1, \mathbb C} \xrightarrow{g_{n-1, \mathbb C}} \dots \xrightarrow{g_{13, \mathbb C}} W_{12, \mathbb C}:=Y_{\mathbb C}
\]
such that $\rho(W_{r, \mathbb C})=r$ for every $r=12,\dots, n$ and such that 
the proper transform of $\Exc(g_{r, \mathbb C})$ on $W_{\mathbb C}$ is $C_{r, \mathbb C}$. 
Indeed, it is easy to check that the push-forward of $C_{i, \mathbb C}$ on $W_{i, \mathbb C}$ is a $(-1)$-curve 
and that the curves $E_{2, \mathbb C}:=(g_{\mathbb C})_*C_{2, \mathbb C}, \dots, E_{12, \mathbb C}:=(g_{\mathbb C})_*C_{12, \mathbb C}$ satisfy  
\[
E_{2, \mathbb C}^2=-6, \,\,\,\,E^2_{m, \mathbb C}=-2, \,\,\,\,E_{i, \mathbb C} \cdot E_{j, \mathbb C}=0
\]
for any $m=3,\dots,12$ and distinct $i,j= 2,\dots, 12$. 
Note also that  $E_{\ell, \mathbb C}$ is isomorphic to $\mathbb P_{\mathbb C}^1$ for any $\ell =2,\dots,12$. 
Let
\[
f_{\mathbb C}\colon Y_{\mathbb C} \to X_{\mathbb C}
\]
be the birational morphism which contracts $E_{2,\mathbb C},\dots,E_{12,\mathbb C}$.
Then $X_{\mathbb C}$ is a projective klt surface with exactly 11 singular points. 
Moreover, since $\rho(W_{\mathbb C})=n$, it follows that 
$\rho(X_{\mathbb C})=1$.

Further, for any invertible sheaf $\mathcal A$ on $\mathcal W$ which is ample over $T$, we have
\[
0>g^*f^*K_X \cdot \mathcal A|_{W}=g_{\mathbb C}^*f_{\mathbb C}^*K_{X_{\mathbb C}} \cdot \mathcal A|_{W_{\mathbb C}}.
\]
Thus, $-K_{X_{\mathbb C}}$ is ample. On the other hand, \cite[Theorem~9.2]{km99}
implies that $X_{\mathbb C}$ admits at most 6 singularities, a contradiction. Thus, (3) holds.  
\end{proof}

\bibliographystyle{amsalpha}
\bibliography{Library}

\end{document}